\documentclass[1p]{elsarticle}

\usepackage{amsmath, amssymb, amsthm, enumerate, anysize, tikz, float, array}

\usetikzlibrary{arrows,shapes}

\newtheorem{notation}{Notation}
 \numberwithin{notation}{section}

 \newtheorem{theorem}{Theorem}[section]
 \newtheorem{corollary}[theorem]{Corollary}
 \newtheorem{lemma}[theorem]{Lemma}
 \newtheorem{proposition}[theorem]{Proposition}

\theoremstyle{remark}
\newtheorem{remark}[theorem]{Remark}
\newtheorem{definition}[theorem]{Definition}
 \newtheorem{example}[theorem]{Example}

 \newcommand{\Aut}{\mathrm{Aut}}
 
 \newcommand{\PGaL}{\mathrm{P\Gamma L}}
 \newcommand{\GL}{\mathrm{GL}}
  \newcommand{\SL}{\mathrm{SL}}
 \newcommand{\PGL}{\mathrm{PGL}}
 \newcommand{\PSL}{\mathrm{PSL}}
 \newcommand{\PG}{\mathrm{PG}}
 \newcommand{\F}{\mathbb{F}}
 \newcommand{\Sym}{\mathrm{Sym}}
 \newcommand{\Gr}{\mathsf{Gr}}
 \newcommand{\Cos}{\textsf{Cos}}
 
 \newcommand{\la}{\langle}
 \newcommand{\ra}{\rangle}
 \newcommand\m{\overline{m}}
\newcommand\ol{\overline}
\newcommand{\Proj}[2]{\mathsf{Proj}_{(#1)}^{\, #2}}

\newcommand{\J}{\overline{J}}

\renewcommand\le\leqslant
\renewcommand\ge\geqslant

\newlength{\mylen}
\newlength{\mylentwo}
\definecolor{darkblue}{rgb}{0,0,0.8}

\begin{document}

\begin{frontmatter}

\title{Triple factorisations of the general linear group and their associated geometries}

\author{Seyed Hassan Alavi}
\ead{alavi.s.hassan@gmail.com}

\author{John Bamberg}
\ead{John.Bamberg@uwa.edu.au}

\author{Cheryl E. Praeger}
\ead{Cheryl.Praeger@uwa.edu.au}

\address{ %
Centre for the Mathematics of Symmetry and Computation,\\
School of Mathematics and Statistics,\\
The University of Western Australia,\\
35 Stirling Highway, Crawley, 6009 W.A.,
Australia.}

\begin{keyword}
triple factorisation, flag-transitive geometry, general linear group, primitive group
\MSC[2008]{20B15, 05E20}
\end{keyword}
% 2000 MSC numbers
% 20B15 Primitive groups
% 05E20 Group actions on designs, geometries and codes

\begin{abstract}
Triple factorisations of finite groups $G$ of the form $G=PQP$ are essential in the study of Lie
theory as well as in geometry.  Geometrically, each triple factorisation $G=PQP$ corresponds to a
$G$-flag transitive point/line geometry such that `each pair of points is incident with at least one
line'. We call such a geometry \emph{collinearly complete}, and duality (interchanging the roles of
points and lines) gives rise to the notion of \emph{concurrently complete} geometries. In this
paper, we study triple factorisations of the general linear group $\GL(V)$ as $PQP$ where the
subgroups $P$ and $Q$ either fix a subspace or fix a decomposition of $V$ as $V_1\oplus V_2$
with $\dim(V_{1})=\dim(V_{2})$.
\end{abstract}

\end{frontmatter}

\section{Introduction}\label{sec:intro}

In linear algebra, the \emph{elimination algorithm}\footnote{This is usually referred to as
  ``Gaussian'' elimination, but perhaps unjustifiably according to \cite{Grcar:2011fk}.} when
starting from the bottom left corner of a matrix $M$ (with entries in a field) and traversing up and
to the right through the matrix, produces a unique factorisation of $M$ as $U_1\pi U_2$ where
$U_1,U_2$ are upper triangular and $\pi$ is a permutation matrix (see \cite{Strang}). This
phenomenon is an instance of the famous \emph{Bruhat/Harish-Chandra decomposition} of a connected
reductive linear algebraic group $G$ as $BNB$, where $B$ is a \textit{Borel subgroup} and $N$ is the
normaliser of a maximal torus contained in $B$. We are concerned here with the case that $G$ is a
general linear group $\GL(n,\mathbb{F})$ so that the \emph{Weyl group} $N/(B\cap N)$ is the
symmetric group $S_n$.  A natural question to ask is whether there are other \emph{triple
  factorisations} of $G$ as $PQP$, for proper subgroups $P$ and $Q$; particularly when $P$ is a
parabolic subgroup.  Finding fruitful canonical forms for invertible matrices is difficult and we
desire prototype subgroup families which provide factorisations $\GL(n,\mathbb{F})=PQP$ for many
different subgroups $P$ and $Q$ in these families.  The factorisations we explore in this paper are
interesting from this perspective, as we shall see that they occur for a wide range of possible
parameters (e.g., dimension, cardinality of $\mathbb{F}$).

Apart from the role that such factorisations have in linear algebra or in the theory of groups (see
the Introduction of \cite{AP09}), there is also an interesting connection with incidence geometry,
whereby the study of flag-transitive point/line geometries involves the study of triple
factorisations of their automorphism groups. A point/line incidence geometry admitting a
flag-transitive group $G$ of automorphisms is geometrically equivalent to a \emph{coset geometry}
$\Cos(G;P,Q)$ (as defined in Section~\ref{ex:cosetgeom}) where $P$ and $Q$ are stabilisers of an
incident point and line, respectively, and this leads to a triple factorisation $G=PQP$ if and only
if
\begin{quote}\label{defn:col-con-intro}
$\Cos(G;P,Q)$ is \emph{collinearly complete}: each pair of distinct points lies on at least one line.
\end{quote}
This necessary and sufficient condition was observed by D. G. Higman and J. E. McLaughlin (see
\cite[Lemma 3]{HigMcL61}) where they introduced in \cite{HigMcL61} the notion of a \emph{geometric
  $ABA$-group} satisfying a more restrictive condition. In particular, they proved that the
(coset) geometries associated with geometric $ABA$-groups
must be linear spaces (see \cite[Proposition 1]{HigMcL61}). Higman and McLaughlin
showed that a geometric $ABA$-group acts primitively (as an automorphism group) on the point set of
the associated linear space (see \cite[Propositions 1-3]{HigMcL61}), and so $A$ is a maximal
subgroup of $G$.  As a generalisation, for a given triple factorisation $G=ABA$, Alavi and Praeger
\cite{AP09} introduced a reduction pathway to the case where $A$ is maximal and core-free, or
equivalently, $G$ acts faithfully as a primitive permutation group on points\footnote{Another
  generalisation appears in the theory of association schemes.  P-H. Zieschang
  \cite{Zieschang:1997fk} proved that if $(X, G)$ is an association scheme (where $X$ is the
  underlying set and $G$ are the classes on $X\times X$), and if there are proper \emph{closed}
  subsets $K$ and $H$ of $G$ satisfying $G=KHK$ under the \emph{complex product}, then $K$ is
  maximal (i.e., if $K$ is a closed subset of $J$, then $J\in \{K,G\}$).}.

There is a wider context in which this analysis plays a substantive part. To study triple
factorisations of finite primitive groups, one needs to study triple factorisations $G=PQP$ of
groups of Lie type where the subgroup $P$ is a maximal subgroup. We can then replace $Q$ by a
maximal overgroup provided $Q$ is not transitive on the right cosets of $P$ in $G$. From the linear
algebraic point of view, a great bulk of the analysis (see \cite{Hassan}) is in handling the case
that $P$ and $Q$ belong to the first two \emph{Aschbacher categories}. That is, these groups either
(i) stabilise a subspace (Category $\mathcal{C}_1$) or (ii) stabilise a direct decomposition of the
underlying vector space (Category $\mathcal{C}_2$).

Our first main result of this paper is a necessary and sufficient condition for the general linear
group to have a triple factorisation by two maximal parabolic subgroups, and it will be proved in
Section~\ref{sec:parabolic}.

\begin{theorem}\label{thm:linear-main}
Let $V$ be a vector space of dimension $n\ge 2$, let $U$ and $W$ be an $m$-subspace and a
$k$-subspace of $V$, respectively, and let $j:=\dim(U\cap W)$. Let $G$ be $\GL(V)$ or $\SL(V)$ (the
\emph{special linear group}).  Then $G=G_UG_WG_U$ if and only if $\max\{0,m+k-n\} \le j\le
\frac{k}{2}+\max\{0,m-\frac{n}{2}\}$.
\end{theorem}

In the statement, by $G_U$ we mean the stabiliser of the subspace $U$ in $G$.  Note that the
parabolic subgroups of $\GL(V)$ (and $\SL(V)$) contain scalars so that the results of these
investigations give rise to triple factorisations of $\PGL(V)$ and $\PSL(V)$ by applying the
\emph{quotient} construction introduced and developed in~\cite[Section 5]{AP09}.

The coset geometry arising from $\GL(V)$ and two of its maximal parabolic subgroups is an
\emph{$(m,k,j)$-projective} space $\Proj{m,k}{j}(V)$ in which the `points' are $m$-subspaces, the
`lines' are $k$-subspaces, with incidence between a point and line when the intersection is a
$j$-subspace (see Section~\ref{sec:proj-mkj}). In Remark~\ref{rem:proj-j}, we observe how these
point/line geometries are related to Desarguesian projective spaces and Grassmannian geometries.  As
collinear completeness of $\Proj{m,k}{j}(V)$ is equivalent to having parabolic triple factorisations
$G=PQP$ with $P$ and $Q$ parabolic (see Lemma~\ref{prop:geom-flag-trans}), we translate Theorem
\ref{thm:linear-main} to the geometric setting.

\begin{quote}\label{thm:proj-main-j}
\textbf{Paraphrase of Theorem \ref{thm:linear-main}.}  Let $V$ be a vector space of dimension $n$.
Then $\Proj{m,k}{j}(V)$ is collinearly complete if and only if $\max\{0,m+k-n\} \le j\le
\frac{k}{2}+\max\{0,m-\frac{n}{2}\}$.
\end{quote}

We also consider triple factorisations of $G=\GL(V)$ where we allow the subgroups to be of
Aschbacher types $\mathcal{C}_1$ and $\mathcal{C}_2$. Within these families we choose subgroups
large enough to allow the possibility that the corresponding geometries can be both collinearly and
concurrently complete. For subgroups $Q$ in $\mathcal{C}_2$, this means (see \cite{Hassan}) that $Q$
is a \emph{bisection subgroup}, that is, $Q$ is the (setwise) stabiliser $G_{\{V_1,V_2\}}$ of a
decomposition $V=V_{1}\oplus V_{2}$ with $\dim(V_{1})=\dim(V_{2})$. We examine
\emph{subspace-bisection} triple factorisations $G=PQP$ and \emph{bisection-subspace} triple
factorisations $G=QPQ$ where $P$ is parabolic and $Q$ is a bisection subgroup. Our approach is to
make use of the associated point/line geometry $\Proj{m,k}{(k_1,k_2)}(V)$. The point set
$\mathbb{P}$ is the set of all $m$-subspaces of $V$, the line set $\mathbb{L}$ is the set of all
bisections $\{V_{1},V_{2}\}$ of $V$ such that $V=V_{1}\oplus V_{2}$ and $\dim(V_{1})
=\dim(V_{2})=k$, and incidence between $U\in \mathbb{P}$ and $\{V_{1},V_{2}\}\in \mathbb{L}$ is
given by $(\dim(U\cap V_{1}),\dim(U\cap V_{2}))= (k_{1},k_{2})$ or $(k_{2},k_{1})$. See Section
\ref{sec:pre} for a detailed description of this flag-transitive geometry. One of the outcomes of
this paper is an exploration of this interesting flag-transitive geometry for the general linear
group, which to the authors' knowledge, has not been studied extensively in the literature.

Theorem~\ref{thm:mainC1C2C1} below is our second major result which will be proved in
Section~\ref{sec:gl-c12}.

\begin{theorem}\label{thm:mainC1C2C1}
Let $V$ be a vector space of dimension $2k$ and let $m$ be a positive integer such that $m<2k$.
Then $\Proj{m,k}{(k_1,k_2)}(V)$ is collinearly complete if and only if
$3k_2\le k+1+m+k_1$ and $(q,m,k,k_1,k_2)\ne (2,1,1,0,0)$.
\end{theorem}

Let $G=\GL(2k,q)$, let $P$ be the stabiliser in $G$ of an $m$-subspace $U$ of $V$, and let $Q$ be
the stabiliser in $G$ of the decomposition $V=V_{1}\oplus V_{2}$.  Suppose $\dim(U \cap V_1)=k_1$
and $\dim(U \cap V_2)=k_2$ with $k_1\le k_2$.  Then Theorem \ref{thm:mainC1C2C1} implies that
$G=PQP$ if and only if $3k_2\le k+1+m+k_1$ and $(q,m,k,k_1,k_2)\ne (2,1,1,0,0)$.

For the dual triple factorisation, the situation is more difficult, and we have the following result (proved in
Lemma \ref{lemma:c21-main-noexist}, Proposition \ref{prop:c21-main-exist}, and Proposition \ref{propn:main_mk}).

\begin{theorem}\label{thm:mainC2C1C2}
Let $V$ be a vector space of dimension $2k$ and let $m$ be a positive integer such that $m\le k$. 
\begin{enumerate}[(i)]
\item If $k_2>m/2$, then $\Proj{m,k}{(k_1,k_2)}(V)$ is concurrently complete if and only if $(q,k)= (2,1)$.
\item 
$\Proj{m,k}{(0,0)}(V)$ is concurrently complete if and only if $(q,k)\ne(2,1), (3,1), (2,2)$.
\end{enumerate}
\end{theorem}

This theorem shows that if $k_2$ is large relative to $m$, then `usually' $\Proj{m,k}{(k_1,k_2)}(V)$ is 
not concurrently complete, while on the other hand, if $k_2=0$ (whence also $k_1=0$), then very often 
$\Proj{m,k}{(0,0)}(V)$ is concurrently complete. Note that Theorem~\ref{thm:mainC2C1C2} also yields
information when $m>k$ on applying the duality result Proposition~\ref{prop:reduction}.

%%%%%%%%%%%%%%%%%%%%%%%%%%%%%%%%%%%%%%
%
%  Rank 2 geometries
%
%%%%%%%%%%%%%%%%%%%%%%%%%%%%%%%%%%%%%%

\section{Rank $2$ geometries}\label{sec:rank2-geom}

In this paper we deal almost exclusively with rank $2$ geometries. Geometries with higher ranks will
appear occasionally, but we only need the formal notion of a geometry in the simplest case. A
\textit{rank $2$ geometry} $\mathcal{G}$, which we sometimes call a point/line geometry, consists of
a set $\mathbb{P}$ of points, a set $\mathbb{L}$ of lines and an incidence relation $\ast$ between
them. A \emph{flag} of $\mathcal{G}$ is an incident point and line pair. We will stipulate that a
rank $2$ geometry has at least one flag, and moreover, we assume that the following three
non-degeneracy properties hold:
\begin{definition}[Non-degeneracy conditions]\label{hyp:geom}
\
\begin{enumerate}[(i)]
\item the sets of points and lines are finite and of size at least two,
\item every point is incident with at least one line,
\item every line is incident with at least one point.
\end{enumerate}
\end{definition}

For a rank $2$ geometry $\mathcal{G}=(\mathbb{P},\mathbb{L},\ast)$, the \emph{dual geometry}
$\mathcal{G}^{\vee}$ of $\mathcal{G}$ is the geometry obtained by interchanging points and lines of
$\mathcal{G}$ and assuming the same incidence relation, that is to say,
$\mathcal{G}^{\vee}=(\mathbb{L},\mathbb{P},\ast)$.  Each rank $2$ geometry $\mathcal{G}
=(\mathbb{P},\mathbb{L},\ast)$ gives rise to a graph $(V,E)$ called its \emph{incidence graph} with
vertex set $V=\mathbb{P}\cup \mathbb{L}$ such that $\{p,\ell\}\in E$ if and only if $p\ast \ell$.
By the definition of a geometry, two elements of the same type are not allowed to be incident, and
so the incidence graph is a bipartite graph.

A \emph{geometry isomorphism} $f$ from $\mathcal{G}_1 =(\mathbb{P}_1,\mathbb{L}_1,\ast_1)$ to
$\mathcal{G}_2 =(\mathbb{P}_2,\mathbb{L}_2,\ast_2)$
is a bijection from the elements $\mathbb{P}_1\cup\mathbb{L}_1$ of $\mathcal{G}_1$ onto the elements 
$\mathbb{P}_2\cup\mathbb{L}_2$ of $\mathcal{G}_2$ such that
\begin{enumerate}
\item[(i)] incidence is preserved: $x \ast_1 y\iff f(x)\ast_2 f(y)$, and
\item[(ii)] points are sent to points, lines are sent to lines: $f(\mathbb{P}_1)=\mathbb{P}_2$ and
  $f(\mathbb{L}_1)=\mathbb{L}_2$.
\end{enumerate}
An automorphism of $\mathcal{G} =(\mathbb{P},\mathbb{L},\ast)$ is a geometry isomorphism of
$\mathcal{G}$ onto itself. The group of all automorphisms of a rank $2$ geometry $\mathcal{G}$,
denoted by $\Aut(\mathcal{G})$, is the \emph{full automorphism group} of $\mathcal{G}$. Note that
$G$ acts on the set of flags of $\mathcal{G}$ via $(p,\ell)^{g}=(p^{g},\ell^{g})$, for all flags
$(p,\ell)$ of $\mathcal{G}$ and $g\in G$. The group $G$ is \emph{flag-transitive} (respectively,
\emph{point-transitive}, \emph{line-transitive}) if $G$ acts transitively on the set of flags
(respectively, the set of points, the set of lines) of $\mathcal{G}$.

Study of a special kind of rank $2$ geometry called a linear space, by Higman and McLaughlin in
\cite{HigMcL61} inspired our work, as mentioned in the introduction. A point/line
geometry is called a \emph{linear space} if any two points are incident with exactly one line, and
more generally it is called a \emph{partial linear space} if each pair of points is incident with at
most one line.

Here we give a formal definition of the completeness properties we study 
that generalise linear spaces (and their duals).

\begin{definition}\label{defn:col-con}
Let $\mathcal{G}=(\mathbb{P},\mathbb{L}, \ast)$ be a rank $2$ geometry. If each pair of distinct
points is incident with at least one line, then $\mathcal{G}$ is said to be \emph{collinearly
  complete}, and if each pair of distinct lines is incident with at least one point, then
$\mathcal{G}$ is said to be \emph{concurrently complete}.
\end{definition}

So a rank $2$ geometry is collinearly complete if and only if its dual geometry is concurrently
complete, and collinear (resp. concurrent) completeness is a geometry isomorphism invariant. 

\subsection{Projective spaces}\label{ex:proj}  Let $V$ be a vector space of finite dimension $n\ge 3$
over a division ring $\F$. The set of all nontrivial proper subspaces of $V$ with incidence given by
symmetrised inclusion is called the \emph{projective geometry} of $V$ and is denoted by $\PG(V)$ or
$\PG(n-1,\F)$. Note also that the full automorphism group of the projective space $\PG(V)$ is the
\emph{projective semilinear group} $\PGaL(n,\F)$ (see \cite[Theorem 3.2]{Buek95-Handbook}).  The
most natural point/line geometry $(\mathbb{P},\mathbb{L},\ast)$ associated with $\PG(V)$ has
$\mathbb{P},\mathbb{L}$ the sets of $1$-subspaces and $2$-subspaces, respectively, with $\ast$ being
inclusion. This geometry is a linear space, and we often refer to it as $\PG(V)$ (abusing the
notation somewhat).

\subsection{Grassmannian geometries}\label{ex:grassmannian}  	
For a vector space $V$ of dimension $n$ over a field $\F$, the set of all $m$-subspaces of $V$ is
known as a \emph{Grassmannian} of $V$ and is denoted by $\Gr_{m}(V)$ (see
\cite[ch. 3]{Taylor92}).  We may also assign a geometric structure to $X:=\Gr_{m}(V)$ and
obtain a rank $2$ geometry $\mathcal{A}_m(V)$ called a \emph{Grassmannian geometry}. The points are
elements of $X$ and the lines are defined as follows: Let $U_{1}$ be an $(m-1)$-subspace and $U_{2}$
be an $(m+1)$-subspace with $U_{1}\subseteq U_{2}$. Then the \emph{Grassmannian line} $L(U_{1},
U_{2})$ defined by $U_{1}$ and $U_{2}$ is the set $\{U\in X \mid U_{1}\subseteq U\subseteq U_{2}
\}$. Note that $\Gr_{m}(V)$ with this endowed incidence structure forms a partial linear
space (see \cite[pp.49-50]{Buek95-Handbook} and \cite[Definition 6.9.1]{BuekCoh-unpub}).

\subsection{$2$-designs}\label{ex:2designs}
Let $v$, $k$, and $\lambda$ be positive integers such that $2\le k\le v$.  A $2$-$(v,k,\lambda)$
design is a rank $2$ geometry $\mathcal{D}:=(\mathbb{X},\mathbb{B},\ast)$ with $|\mathbb{X}|=v$,
$\mathbb{B}$ is a collection of $k$-element subsets (called \emph{blocks} of $\mathbb{X}$), the
incidence relation $\ast$ is simply given by the inclusion relation, and each pair of points of
$\mathbb{X}$ is contained in exactly $\lambda$ blocks of $\mathbb{B}$.  Note that the projective
space $\PG(n-1,\mathbb{F})$, where $\mathbb{F}$ is a finite field of order $q$, is a
$2$-$(\frac{q^{n}-1}{q-1},q+1,1)$ design with $\mathbb{X}$ the set of all $1$-subspaces of $V$, and
$\mathbb{B}$ the set of all $2$-subspaces of $V$.

\subsection{Coset geometries}\label{ex:cosetgeom}
For a group $G$ and a subgroup $H$ of $G$, let $\Omega_{H}$ be the set of right cosets of $H$ in
$G$. Then the group $G$ acts on $\Omega_{H}$ by right multiplication and the kernel of the action is
the \emph{core} $\cap_{g\in G}H^g$ of $H$ in $G$. Let $G$ be a group with $P$ and $Q$ proper
subgroups of $G$. Set $\mathbb{P}:=\Omega_{P}$ and $\mathbb{L}:=\Omega_{Q}$. We say that the
elements $p:=Px\in \mathbb{P}$ and $\ell:=Qy\in \mathbb{L}$ are incident, and we write $p\ast \ell$,
if and only if $ Px\cap Qy\neq \varnothing$.  One can check that the non-degeneracy properties
(Definition \ref{hyp:geom}) hold when $P\cap Q$ is a proper subgroup of finite index in both $P$ and
$Q$. This geometry is denoted $\Cos(G;P,Q)$ and is called the \emph{coset geometry} associated with
the group $G$ with fundamental subgroups $P$ and $Q$. In particular, $G$ is a flag-transitive group
of automorphisms of this geometry and we have the following converse:

\begin{lemma}[{\cite[Lemmas 1 and 3]{HigMcL61}}]\label{prop:geom-flag-trans}
Let $\mathcal{G}$ be a rank $2$ geometry and $G\le \Aut(\mathcal{G})$. Then $G$ acts transitively on
the flags of $\mathcal{G}$ if and only if $\mathcal{G}\cong \Cos(G;P,Q)$ where $P$ is the stabiliser
of a point $p$ and $Q$ is the stabiliser of a line $\ell$ incident with $p$. Moreover, $\Cos(G;P,Q)$
is collinearly complete (resp. concurrently complete) if and only if $G=PQP$ (resp. $G=QPQ$).
\end{lemma}

\begin{remark}\label{rem:coset-geom}
Note that in our study of triple factorisations of finite groups $G$, we naturally exclude the case
where $P\subseteq Q$ and $Q\subseteq P$.  These cases, in the language of triple factorisations,
give rise to either trivial triple factorisations (if $G$ equals one of $P$ or $Q$), or no triple
factorisation (if $P$ and $Q$ are both proper subgroups). As coset geometries, they are also rather
trivial as either each point is on just one line (if $P\subseteq Q$) or each line is incident with
just one point (if $Q \subseteq P$), and the incidence graph of the geometry is disconnected if $P,
Q$ are both proper subgroups.
\end{remark}

If $G=PQ$, then $|G:P|=|Q:P\cap Q|$ and $|G:Q|=|P:P\cap Q|$, so every point of $\Cos(G;P,Q)$ is
incident with all lines of $\Cos(G;P,Q)$, and every line is incident with all points, and hence the
incidence graph of $\Cos(G;P,Q)$ is complete bipartite.  Although, by
Lemma~\ref{prop:geom-flag-trans}, each triple factorisation naturally introduces a coset geometry,
not every coset geometry gives rise to a triple factorisation. For example, let
$G=\Sym(\{1,2,\ldots, 5\})$, $P=\langle (4,5) \rangle$ and $Q=\langle (1,2,3) \rangle$. Then $G\neq
PQP$ while $\Cos(G;P,Q)$ is a $G$-flag-transitive rank $2$ geometry.

\subsection{Buekenhout geometries with point-diameter at most $3$}

A collinearly complete rank $2$ geometry has an associated Buekenhout diagram $\Gamma$ with
point-diameter at most $3$ (see \cite{Buekenhout:1979fk}). This means that there are only five
possible values\footnote{We thank Dimitri Leemans for pointing out this fact.} for the canonical
parameters of $\Gamma$; the point-diameter $d_p$, gonality $g$, and line-diameter $d_\ell$.  It
turns out that $(d_p,g,d_\ell)$ can only be one of the following:

\begin{description}
\item[$\mathbf{(2,2,2)}$:] These geometries are simply the \emph{generalised di-gons} whose
incidence graphs are complete bipartite. For the automorphism group $G$, we have a degenerate
factorisation $G=PQ$.

\item[$\mathbf{(3,3,3)}$ and $\mathbf{(3,3,4)}$:] These two cases are flag-transitive linear spaces
  that we referred to earlier in the work of Higman and McLaughlin, and have been classified up to
  the \emph{one-dimensional affine} case \cite{BueDDKLS90}.  The parameters $(3,3,3)$ yield
  projective planes, while the case $(3, 3, 4)$ with point-order $s_p = 1$ corresponds to complete
  graphs and $2$-transitive group actions (which are known completely by the classification of the
  finite $2$-transitive permutation groups).

\item[$\mathbf{(3,2,3)}$ and $\mathbf{(3,2,4)}$:] These cases cover the collinearly complete
  geometries we are studying in this paper.
\end{description}

%%%%%%%%%%%%%%%%%%%%%%%%%%%%%%%%
%
% $(m,k,j)$-projective spaces
%
%%%%%%%%%%%%%%%%%%%%%%%%%%%%%%%%

\subsection{The $m$-subspaces versus $k$-subspaces geometries $\Proj{m,k}{j}(V)$}\label{sec:proj-mkj}

 Let $n$, $m$, $k$ and $j$ be positive integers with $m,k<n$.  As in Section~\ref{ex:grassmannian},
 we denote by $\Gr_{m}(V)$, the set of all $m$-subspaces of $V$. Whenever $j$ satisfies
\begin{align}\label{eq:admis}
    \max\{0,m+k-n\}\le j\le \min\{m,k\}
\end{align}
we define the $(m,k,j)$-projective space $\Proj{m,k}{j}(V)$ as the rank $2$ geometry with point set
$\mathbb{P}:=\Gr_{m}(V)$ and line set $\mathbb{L}:=\Gr_{k}(V)$, and with incidence
relation $\ast^j$ given by $\dim(U\cap W)=j$, for $U\in \mathbb{P}$ and $W\in \mathbb{L}$.

Let $V$ be a vector space of dimension $n$ over a field $\F$. Then $V$ may be viewed as the row space
$\F^{n}$. Define the `$\perp$-map' (in words, `perp map') on subspaces $U$ of $V$ by
\begin{align}\label{eq:uperp}
U^{\perp}=\{v\in V\mid u\cdot v=0, \hbox{ for all } u\in U \}
\end{align}
where ``$\cdot$'' denotes the scalar product on $V$.  Here $\dim(U^{\perp})=n-\dim(U)$.

\begin{remark}\label{rem:proj-j}\
\begin{enumerate}[{ \ (a)}]
   \item Set $j_{0}=\min\{m,k\}$. In this case the incidence relation is symmetrised inclusion, and
     hence $\Proj{m,k}{j_{0}}(V)$ is the $\{m,k\}$-truncation of the projective geometry
     $\PG(V)$. In particular, $\Proj{1,2}{1}(V)$ with $\dim V\ge 3$ is the natural rank $2$ geometry
     for the projective space $\PG(V)$ mentioned in Section~\ref{ex:proj}.

  \item The flags of $\Proj{m,m}{m-1}(V)$ give rise to the Grassmannian geometry
    $\mathcal{A}_{m}(V)$ of Section~\ref{ex:grassmannian} via the map
      \begin{equation*}
      (U,W)\mapsto L(U\cap W,U+W)
      \end{equation*}
      for all flags $(U,W)$ of $\Proj{m,m}{m-1}(V)$. Note that, for every flag $(U,W)$ of
      $\Proj{m,m}{m-1}(V)$, $\dim(U\cap W)=m-1$. We have already seen that the Grassmannian geometry
      is a partial linear space and for $U,W\in \Gr_{m}(V)$ with $\dim(U\cap W)=m-1$, there exists
      exactly one Grassmannian line $L(U\cap W,U+W)$ incident with both $U$ and $W$. It is not
      difficult to see that $\Proj{m,m}{m-1}(V)$ is collinearly complete if and only if the
      collinearity graph of $\mathcal{A}_{m}(V)$ has diameter $2$. Since the diameter of
      $\mathcal{A}_{m}(V)$ is $\min\{m,n-m\}$ (see \cite[Lemma 2.5]{BlokCooperstein}), we see that
      $\Proj{m,m}{m-1}(V)$ is collinearly complete precisely when $m=2$ or $m=n-2$. This observation
      is in line with our results in Theorem~\ref{thm:linear-main}. Notice that when $\dim V=4$, the
      Grassmannian geometry is (famously) isomorphic to the \textit{Klein quadric}.
      \item If $m=k$, then the possible incidence relations $\ast^j$ form an association scheme known as the
\textit{Grassmann scheme} (see \cite[\S 16.3]{GodsilBook}).
\end{enumerate}

\end{remark}

\begin{proposition}\label{prop:proj-flag}
Let $V$ be a vector space of dimension $n\ge 2$ over a field $\F$, let $m,k<n$ and let $j$ be a
non-negative integer satisfying $\max\{0,m+k-n\}\le j\le \min\{m,k\}$. Then
\begin{enumerate}
   \item[(a)] $\Proj{m,k}{j}(V)$ has at least one flag and $\GL(V)$ acts transitively on its flags.
  \item[(b)] $\Proj{m,k}{j}(V)$ is collinearly complete if and only if
    $\Proj{\bar{m},\bar{k}}{\bar{j}}(V)$ is collinearly complete (where $\bar{m}= n-m, \bar{k}=n-k,
    \bar{j} =n-m-k+j$). 
\end{enumerate}
\end{proposition}

\begin{proof}
\noindent\textbf{(a)}\ Let $\{e_{1},\ldots, e_{n}\}$ be a basis for $V$, and take $U=\langle e_{1},\ldots, e_{m}
  \rangle$. Since $m+k-j\le n$, there exists a $k$-subspace $W$ of $V$ intersecting $U$ in a
  subspace of dimension $j$; for example, $W:=\langle e_{1},\ldots, e_{j},$ $e_{m+1},\ldots,
  e_{k+m-j} \rangle$. So $\Proj{m,k}{j}(V)$ has flags.

 Suppose $f:=(U,W)$ and $f':=(U',W')$ are flags of $\Proj{m,k}{j}(V)$ . Then $\dim(U\cap
 W)=\dim(U'\cap W')=j$. Since $G:=\GL(V)$ is transitive on the set $\Gr_{m}(V)$ of all
 $m$-subspaces of $V$, there exists $x\in G$ such that $U^{x}=U'$. Similarly, since $G_{U'}$ is
 transitive on $\Gr_j(U')$, there also exists $y\in G_{U'}$ such that $(U\cap
 W)^{xy}=U'\cap W'$. Moreover, there is $z\in G_{\{U',U'\cap W'\}}$ which maps $W^{xy}/(U'\cap W')$
 to $W'/(U'\cap W')$. Therefore, $g:=xyz\in G$ maps $f$ to $f'$ and $(U\cap W)^{g}=U'\cap W'$.

\noindent\textbf{(b)}\ Define a map $f$ from the elements of $\Proj{m,k}{j}(V)$ to the elements of
$\Proj{\bar{m},\bar{k}}{\bar{j}}(V)$ by $f(U):=U^\perp$ for all $U\in \Gr_{m}(V)\cup \Gr_{k}(V)$. It
is straight-forward to prove that $f$ is a bijection, and maps $\Gr_m(V)$ to $\Gr_{\m}(V)$ (`points
to points') and maps $\Gr_k(V)$ to $\Gr_{\bar{k}}(V)$ (`lines to lines'). So in order to show that
$f$ is a geometry isomorphism, we only need to show that incidence is preserved. Suppose we have
$U\in\Gr_{m}(V)$ and $W\in\Gr_{k}(V)$ which are incident in $\Proj{m,k}{j}(V)$; that is, $\dim U\cap
W=j$. Since $U^\perp \cap W^\perp= (U+W)^\perp $, we have
\[
\dim(U^\perp\cap W^\perp)=2k-\dim( U+W)=2k-\dim U-\dim W+\dim(U\cap W)=k-m+j=\bar{j}.\
\]
Therefore, $U^\perp$ and $W^\perp$ are incident in $\Proj{\bar{m},\bar{k}}{\bar{j}}(V)$, and we have
shown that $f$ is a geometry isomorphism. It then follows that $\Proj{m,k}{j}(V)$ is collinearly
(respectively, concurrently) complete if and only if $\Proj{\bar{m},\bar{k}}{\bar{j}}(V)$ is
collinearly (respectively, concurrently) complete.
\end{proof}

Proposition~\ref{prop:proj-flag}~(b) allows us, for example, to reduce to the case
$m\le n/2$, when studying the collinear completeness of  $\Proj{m,k}{j}(V)$.

\subsection{The subspace--bisection geometries $\Proj{m,k}{J}(V)$}

For this geometry, denoted by $\Proj{m,k}{J}(V)$, the point set $\mathbb{P}$ is the set $\Gr_{m}(V)$
of all $m$-subspaces of $V$, the line set $\mathbb{L}$ is the set of all bisections
$\{V_{1},V_{2}\}$ of $V$ such that $V=V_{1}\oplus V_{2}$ and $\dim(V_{1}) =\dim(V_{2})=k$, and
incidence between $U\in \mathbb{P}$ and $\{V_{1},V_{2}\}\in \mathbb{L}$ is given by $(\dim(U\cap
V_{1}),\dim(U\cap V_{2}))= (k_{1},k_{2})$ or $(k_{2},k_{1})$. The authors have not found any
literature where this flag-transitive geometry for $\GL(V)$ has been studied. We study these
goemetries because
\begin{itemize}
\item they are an infinite family of flag-transitive geometries for the general linear groups,
\item they are a large class of rank $2$ geometries with canonical parameters $(3,2,3)$ and $(3,2,4)$, and
\item they provide triple factorisations for $\GL(V)$ for a large range of parameters ($m$, $k$, $J$, and $q$).
\end{itemize}

Example~\ref{ex:ssp-bis} describes a family of such geometries which occurs naturally in incidence
geometry, namely the cases where $m=k=2$ can be interpreted as geometries related to the Klein
quadric, and in particular its collinearity graph. The \emph{collinearity graph} of a point/line
geometry has as vertices the points of the geometry, and two points are adjacent if they are
incident with a common line.

\begin{example}\label{ex:ssp-bis}
Suppose $m=k=2$, so the points are the $2$-subspaces of $V(4,\mathbb{F})$, and the lines are the
bisections of this vector space into complementary $2$-subspaces. A complementary pair of
$2$-subspaces can be viewed as a pair of disjoint lines of the associated projective space
$\PG(3,\F)$, and so via the Klein correspondence, each bisection corresponds to a pair of
non-collinear points of the Klein quadric $\mathsf{Q}^+(5,\F)$. Suppose now that $\F$ is finite.
Then a pair of non-collinear points of the Klein quadric spans a non-degenerate line of $+$-type in
$\mathsf{Q}^+(5,\F)$, and the perps of such lines are $3$-dimensional hyperbolic quadric sections.
Thus the lines of our rank $2$ geometry correspond to the $3$-dimensional hyperbolic quadric
sections of $\mathsf{Q}^+(5,\F)$, and the points are simply the points of $\mathsf{Q}^+(5,\F)$.
There are ostensibly four different rank $2$ geometries corresponding to different incidence
relations, described in Table~\ref{tbl:ssp-bis}.

\begin{table}[H]
\begin{tabular}{p{9cm} | p{6cm}}
\hline
Incidence relation: point $P$ versus hyperbolic quadric section $H$ of $\mathsf{Q}^+(5,\F)$ &Collinearity graph\\
\hline
$P+H$ is a non-degenerate hyperplane& Complete\\
$P$ does not lie in $H^\perp$ and $P^\perp\cap H^\perp$ is a point of $\mathsf{Q}^+(5,\F)$& Complete\\
$P$ lies in $H^\perp$&Complement of the collinearity graph of $\mathsf{Q}^+(5,\F)$ (which is strongly regular) \\
$P$ lies in $H$& Complete\\
\hline
\end{tabular}
\caption{The four incidence relations for Example~\ref{ex:ssp-bis} and their collinearity graphs.}\label{tbl:ssp-bis}
\end{table}

\end{example}

%%%%%%%%%%%%%%%%%%%%%%%%%%%%%%%%
%
% Groups with a BN-pair
%
%%%%%%%%%%%%%%%%%%%%%%%%%%%%%%%%

\subsection{Thin geometries and groups with a $BN$-pair}

We follow the standard notation introduced by Tits \cite{Titsbook}. A \textit{$BN$-pair} in a group
$G$ is a pair of subgroups $B$ and $N$ such that $G=\langle B,N\rangle$, the intersection $T=B\cap
N$ is a normal subgroup of $N$, the factor group $W=N/T$ is generated by involutions
$S:=\{s_1,\ldots, s_n\}$, and the following holds for all $w\in W$ and all $s\in S$:
$$sBw\subset BwB\cup BswB\quad\text{ and }\quad sBs\ne B.$$ The subgroup $B$ is called the
\textit{standard Borel subgroup}, and any $G$-conjugate of $B$ is a \textit{Borel subgroup} of
$G$. The group $W$ is the \textit{Weyl group}, and $T$ is the \textit{Cartan subgroup}. The
subgroups containing $B$ are the \textit{standard parabolic subgroups} of $G$. If $J\subset
\{1,\ldots,n\}$, then the group $W_J:=\langle s_j:j\in J\rangle$ is a \textit{Weyl-parabolic}, and
we write $N_J$ for the subgroup of $N$ containing $T$ such that $N_J/T = W_J$. The \textit{Bruhat
  Decomposition Theorem} states that $G=BNB$, and a similar property holds for parabolic subgroups:
for each standard parabolic subgroup $P$, there exists a unique subset $J$ of $\{1,\ldots,n\}$ such
that $P=BN_JB$.  We will often write $P_J$ for the standard parabolic $BN_JB$, and indeed abuse
notation and write $P_J=BW_JB$.

\begin{lemma}[{\cite[pp. 28/29]{Bourbaki}}]\label{Bourbaki}
Let $K$ and $L$ be subsets of $\{1,..,n\}$. Then there is a bijection between the double coset
spaces $W_K \backslash W / W_L$ and $P_K \backslash G /P_L$ given by
$$
\Lambda \to B\Lambda B,\quad \text{for all }\Lambda\in W_K \backslash W / W_L.
$$
\end{lemma}

In this paper, we are only interested in the cases where the Weyl group is of \textit{type}
$\mathsf{A}_n$. The $\mathsf{B}_n/\mathsf{C}_n$ case will be explored in a forthcoming paper
\cite{ABP2010}. A parabolic triple factorisation for the Weyl group gives rise to a triple
factorisation for $G$.

\begin{lemma}\label{Weylfactorisation}
Let $G$ be a group with a $(B,N)$-pair and let $A_1$ be a standard parabolic subgroup of $G$.  Let
$W$ be the Weyl group $N/(B\cap N)$ of $G$.  Let $W_1=(A_1\cap N)/(B\cap N)$ and suppose there
exists a subgroup $W_2\le W$ such that $W=W_1W_2W_1$.  Then $G=A_1N_2A_1$ where $N_2$ is the
preimage of $W_2$ in $N$.
\end{lemma}

\begin{proof}
Let $N_1$ and $N_2$ be the preimages of $W_1$ and $W_2$ in $N$, respectively. Then $N_1=A_1\cap N$,
and $N=N_1N_2N_1$. It follows that
$N\subseteq N_1BN_2BN_1$ and so
$$G=BNB \subseteq BN_1B N_2 BN_1B=A_1N_2A_1.$$
Therefore, $G=A_1N_2A_1$.
\end{proof}

Thus, we have the following corollary of Lemma \ref{Bourbaki}. We point out here that the proofs of
the result below will be in terms of the group $\GL(V)$, however, they equally apply to $\SL(V)$
since the $BN$-pair counterpart for $\SL(V)$ has the same properties as those needed for
$\GL(V)$. See \cite[Section 9.5]{Garrett:1997fk}. In view of Corollary~\ref{cor:c11-orbits}, we say
that an $m$-subspace $U$ and a $k$-subspace $W$ are \emph{$j$-related} if $\dim(U\cap W)=j$. We use
this notion in particular in Section~\ref{sec:pre} in the case $m=k$.

\begin{corollary}\label{cor:c11-orbits}
Let $G$ be $\GL(V)$ or $\SL(V)$.  Let $P_m$ be the stabiliser in $G$ of an $m$-subspace $U$ of a
vector space $V$. Then the orbits of $P_m$ on $k$-spaces are characterised by their dimension $j$ of
intersection with $U$, where $\max\{0,m+k-n\}\le j\le\min\{m,k\}$.
\end{corollary}

\begin{proof}
As $G$ is transitive on $\Gr_m(V)$, we may assume that $P_m$ is the standard parabolic
$P_M$ with $M=\{1,\dots,m\}$.  Let $\mathcal{O}=\Gr_k(V)$, and let $P_k=P_K$ be the
standard parabolic stabilising a $k$-subspace corresponding to a $k$-subset $K$ of $\{1,\dots,n\}$.
Now the elements of $\mathcal{O}$ are in bijection with the left cosets of $P_K$ in $G$, and so the
orbits of $P_M$ on $\mathcal{O}$ are determined by the double cosets in $P_K\backslash G/P_M$. (In
fact, each orbit is the set of left cosets contained in a certain double coset of $P_K\backslash
G/P_M$.)  The Weyl group $W$ of $G$ is the symmetric group $S_n$; the Weyl-parabolic $W_M$
corresponding to $P_M$ is the full stabiliser in $W$ of the $m$-subset $M$ of $\{1,\ldots,n\}$, and
the Weyl parabolic $W_K$ corresponding to $P_K$ is the stabiliser of $K$ in $W$.  Now by
Lemma~\ref{Bourbaki}, the double cosets in $P_K\backslash G/P_M$ are in 1--1 correspondence with the
double cosets in $W_K\backslash W/W_M$. In turn, the double cosets in $W_K\backslash W/W_M$
correspond to the $W_M$-orbits on $k$-subsets of $\{1,\dots,n\}$, and these are precisely the
families of $k$-subsets that intersect $M$ in a set of some fixed size.  The possible sizes of the
intersections with $M$ of a $k$-subset are identical with the possible intersection dimensions of
$k$-subspaces with $U$, namely the integers $j$ satisfying $\max\{0, m + k - n\}\le j\le\min\{m,
k\}$. Thus the $P_M$-orbits on $\mathcal{O}$ are precisely the subsets of $\mathcal{O}$ consisting
of the $k$-subspaces meeting $U$ in a subspace of some fixed dimension.
\end{proof}

\section{Proof of Theorem~\ref{thm:linear-main} }\label{sec:parabolic}

The proof of our first main result will rely in part on triple factorisations of the Weyl group
$S_n$ of $\GL(n,q)$.  In Example~\ref{ex:con-nodes} below, we construct flag-transitive geometries
that include examples which are concurrently complete. In particular, some of these geometries are
$2$-designs that are not symmetric $2$-designs.  The points and lines of these geometries are
$m$-subsets and $k$-subsets of an $n$-set, respectively.  By replacing all these subsets by their
complements, if necessary, we lose no generality by assuming that $m$ is at most $n/2$.

\begin{example}\label{ex:con-nodes}
Let $G=S_{n}$ acting on $\Omega=\{1,\dots,n\}$, and suppose that $1\le m\le n/2$ and $1\le k<n$.
Let $\mathbb{P}$ be the set of $m$-subsets of $\Omega$, and let $\mathbb{L}$ be the set of
$k$-subsets of $\Omega$.  Let $j$ be a non-negative integer such that $\max\{0,m+k-n\}\le
j\le\min\{m,k\}$.  We stipulate that an element of $\mathbb{P}$ is incident with an element of
$\mathbb{L}$ if they intersect in $j$ elements, and we thereby obtain a rank $2$ geometry
$\mathcal{X}$. (The reader may verify that there exists a flag.)  Now consider the element
$M_1:=\{1,2,\ldots,m\}$ of $\mathbb{P}$ and another $m$-subset $M_2$ distinct from $M_1$. The
$G$-orbits on pairs of $m$-subsets correspond to the possible sizes $t$ of $M_1\cap M_2$, namely
$0\le t\le m-1$. In order for $\mathcal{X}$ to be collinearly complete, there must be enough room,
for each such $t$, for there to be a $k$-subset $\ell$ meeting both $M_1$ and $M_2$ in $j$ elements,
as Figure \ref{fig:enoughroom} demonstrates.

%%%%%%%%%%%%%%%%%%%
%
%  Figure
%
%%%%%%%%%%%%%%%%%%%
\begin{tabular}{m{0.4\textwidth}m{0.5\textwidth}}
\begin{minipage}{0.4\textwidth}
\begin{figure}[H]\footnotesize
  \caption{A line $\ell$ incident with the points $M_1$ and $M_2$, where $|M_1 \cap M_2| =
    t$.}\label{fig:enoughroom}

  \centering
\begin{tikzpicture}[scale = 0.7]

    \begin{scope}[shift={(3cm,-5cm)}, fill opacity=1]
        \draw  (0,0) circle (2cm);
        \draw  (1.5,-2) circle (2cm);
        \draw (0:3cm) circle (2cm);
        \node at (1.5,0.35) {$t-i$};
        \node at (1.5,-0.8) {$i$};
        \node at (0.5, -1.2) {$j-i$};
         \node at (2.5, -1.2) {$j-i$};
         \node at (-0.4,0.8) {$m-t-j+i$};
         \node at (3.3,0.8) {$m-t-j+i$};
         \node at (1.5,-2.5) {$k-2j+i$};

        \node at (-2.6,0) {$M_1$};
        \node at (5.6,0) {$M_2$};
        \node at (3.7,-3.2) {$\ell$};

    \end{scope}
\end{tikzpicture}
\end{figure}
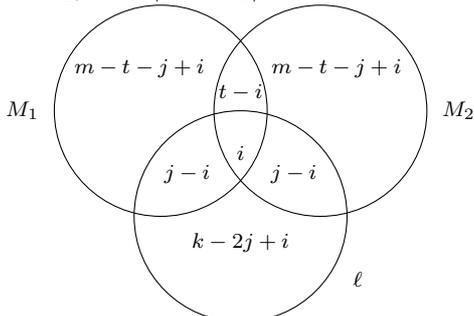
\end{minipage}&
The set $\ell$ exists if and only if, for every possible $t$, we can choose $i=|M_1\cap M_2\cap
\ell|$ such that the complement of $M_1\cup M_2$ contains at least $k-2j+i$ elements, that is to
say, $0\le k-2j+i\le n- (|M_1|+|M_2|-|M_1\cap M_2|) =n-2m+t.$ We claim that this is possible for
every possible $t$-value, or equivalently $\mathcal{X}$ is collinearly complete, if and only if
$0\le k-2j\le n-2m$. The case $t=0$ affirms that if $0\le k-2j\le n-2m$ fails, then $\mathcal{X}$ is
not collinearly complete.  The converse will follow from Lemma \ref{lem:themdiagonals}(iii) below.
\end{tabular}
\end{example}

We will also be using the following lemma, in particular, parts (iii) and (iv), to prove Theorem
\ref{thm:linear-main}.

\begin{lemma}\label{lem:themdiagonals}
Let $N$ be the set $\{1,\ldots,n\}$, let $1\le m\le n/2$, let $\max\{0,m+k-n\}\le j\le \min\{m,k\}$,
and let $1\le k<n$. If $M_1$ and $M_2$ are distinct $m$-subsets of $N$, and $2j\le k$, then there
exists a $(n-2m+2j)$-subset $P$ intersecting $M_1$ and $M_2$ in $j$-subsets such that:
\begin{enumerate}[(i)]
\item If $0\le k-2j\le n-2m$, then there is a $k$-subset $K$ of $P$ such that
$|K\cap M_1|=|K\cap M_2|=j$;
\item If $k-2j>n-2m$, then there is a partition of $N\backslash P$ into $(k-2j)-(n-2m)$ mutually
  disjoint subsets, each of size at least $2$.
\end{enumerate}
\end{lemma}

\begin{proof}
Let $t:=|M_1\cap M_2|$.
Without loss of generality, we may assume that
\begin{align*}
M_1&:=\{1,2,\ldots,m\},\\
M_2&:=\{m-t+1,m-t+2,\ldots,2m-t\}.
\end{align*}

Suppose first that $j \le m-t$. Then let
\begin{align*}
P_1&=\{1,\ldots,j\}\subseteq M_1\setminus M_2,\\
P_2&=\{2m-t-j+1,\ldots, 2m-t\}\subseteq M_2\setminus M_1.
\end{align*}
and note that $|P_1\cup P_2|=2j$.  In the complement of $M_1\cup M_2$, there are $n-2m+t$ elements,
and so we can find $n-2m$ elements to adjoin to $P_1\cup P_2$ to create an $(n-2m+2j)$-subset 
intersecting $M_1$ and $M_2$ in $j$-subsets. In particular, $P:=\{2m-t+1,\ldots, n-t\}$ is such a set.

Now suppose that $j > m-t$.  If $j \le t$, then set $P_1:=\{m-t+1,\ldots, m-t+j \}$ and note that
$|P_1|=j$ and $P_1\subset M_1\cap M_2$. The complement of $M_1\cup M_2$ has $n-2m+t$ elements, and
so we can find $n-2m+j$ elements there, to create a set of size $n-2m+2j$ with
$P_1$. Then $P:=P_1\cup \{2m-t+1,\ldots, n+j-t\}$ is an $(n-2m+2j)$-subset 
intersecting $M_1$ and $M_2$ in $j$-subsets.

Finally if $j>\max\{m-t,t\}$, then set $P_1:=\{m+1-j,\ldots, m+j-t\}$ and note that
$|P_1|=2j-t$. Again, the complement of $M_1\cup M_2$ has $n-2m+t$ elements, and so if we set
$P:=P_1\cup \{2m-t+1,\ldots, n\}$, then we will have a set $P$ of size $n-2m+2j$ 
intersecting $M_1$ and $M_2$ in $j$-subsets

Now we prove part (iii), so suppose $0\le k-2j\le n-2m$. From the results above, we have a set $P$
of size $n-2m+2j$ such that $|P\cap M_1|=|P\cap M_2|=j$. In each case there exists a $k$-subset $K$
of $P$ such that $K\cap M_i=P\cap M_i$ for $i=1,2$, proving condition (iii).

Finally, we prove (iv) and we assume now that $k-2j>n-2m$. In all situations above, $N\backslash P$
has $2m-2j$ elements.  So we can partition $N\backslash P$ into $m-j$ subsets of size $2$.  Merging
some parts of this partition if necessary leads to a partition of $N\backslash P$ into
$(k-2j)-(n-2m)$ sets of size at least $2$.
\end{proof}

Using Example~\ref{ex:con-nodes}, Lemma \ref{Weylfactorisation}, Proposition~\ref{prop:proj-flag},
we give a proof of Theorem~\ref{thm:linear-main}.

\begin{proof}[Proof of Theorem~\ref{thm:linear-main}]
Note that the lower bound on $j$ follows automatically from \eqref{eq:admis}, and so it suffices to
show that $\Proj{m,k}{j}(V)$ is collinearly complete if and only if $2j\le k+\max\{0,2m-n\}$.  First
we consider the case $m\le n/2$. We must prove that $\Proj{m,k}{j}(V)$ is collinearly complete if
and only if $2j\le k$.  Suppose first that $\Proj{m,k}{j}(V)$ is collinearly complete.  Let $U_1,
U_2$ be disjoint $m$-subspaces.  Then there exists a $k$-subspace $W$ such that $\dim(U_1\cap
W)=\dim(U_2\cap W)=j$.  So, in particular, the direct sum of $U_1\cap W$ and $U_2\cap W$ is a
$2j$-subspace of $W$ and hence $2j\le k$.
Conversely suppose that $2j\le k$.  Then by Example~\ref{ex:con-nodes} and Lemma
\ref{Weylfactorisation}, if $0\le k-2j\le n-2m$ then $\Proj{m,k}{j}(V)$ is collinearly complete.  So
suppose $k-2j>n-2m$, and let $\{e_1,\ldots, e_n\}$ be the canonical basis for $V$. Assume we have
two $m$-subspaces $M_1$ and $M_2$ with $\dim(M_1\cap M_2)=t$.  (The possible values for $t$ are as
in Corollary \ref{cor:c11-orbits}.)  Without loss of generality, $M_1=\langle e_1,\ldots,
e_m\rangle$ and $M_2=\langle e_{m-t+1},\ldots, e_{2m-t}\rangle$.  By Lemma \ref{lem:themdiagonals},
there exists a subset $P$ of $N:=\{1,\ldots, n\}$ of size $n-2m+2j$ such that $P$ intersects each of
$\{1,\ldots, m\}$ and $ \{m-t+1,\ldots, 2m-t\}$ in a $j$-subset, and there is a partition
$\mathcal{P}$ of $N\backslash P$ into $(k-2j)-(n-2m)$ subsets, each of size at least $2$.  Let
$S_1=\langle e_i: i\in P \rangle$, and note that $S_1$ has dimension $n-2m+2j$ and meets $M_1$ and
$M_2$ in $j$-dimensional subspaces.  Let $S_2= \langle \sum_{i\in \pi} e_i: \pi \in
\mathcal{P}\rangle$ and note that $S_2$ has dimension $(k-2j)-(n-2m)$ and intersects each of $S_1$,
$M_1$, and $M_2$ trivially, since each part of $\mathcal{P}$ has size at least $2$. It then follows
that $S:=S_1\oplus S_2$ is a $k$-dimensional subspace of $V$ that intersects $M_1$ and $M_2$ in
$j$-dimensional subspaces. Hence $\Proj{m,k}{j}(V)$ is collinearly complete.  This proves
Theorem~\ref{thm:linear-main} for $m\le n/2$.

Now suppose that $m>n/2$. By Proposition~\ref{prop:proj-flag}, $\Proj{m,k}{j}(V)$ is collinearly
complete if and only if $\Proj{\bar{m},\bar{k}}{\bar{j}}(V)$ is collinearly complete, where
$\bar{m}= n-m, \bar{k}=n-k, \bar{j} =n-m-k+j$. Note that $\bar{m}< n/2$, and as we have just proved,
$\Proj{\bar{m},\bar{k}}{\bar{j}}(V)$ is collinearly complete if and only if $2\bar{j} \le \bar{k}$,
that is, $2(n-m-k+j)\le n-k$, or equivalently, $2j\le k+2m-n$. This completes the proof.
\end{proof}

%%%%%%%%%%%%%%%%%%%%%%%%%%%%%
%
% Subspace-bisection geometries:
%
%%%%%%%%%%%%%%%%%%%%%%%%%%%%%

\section{Subspace-bisection geometries}\label{sec:pre}

\subsection{The reduction to the case $m\le k$}\label{sec:reduction}

In this subsection, we give a geometry isomorphism between certain subspace-bisection geometries
that will be used to reduce the proof of Theorems~\ref{thm:mainC1C2C1} and~\ref{thm:mainC2C1C2} to
the cases $m\le k$. We first establish some notation and definitions that we will use throughout
this section.

\begin{notation}\label{notation:c12-0}
Let $V$ be a vector space of dimension $n=2k$ over a finite field $\F$ of size $q$,
let $V=V_1\oplus V_2$ with $\dim(V_{1})=\dim(V_{2})=k$, and let $G=\GL(V)$ and
$Q=G_{\{V_1,V_2\}}$. Let $m$ be an integer such that
$1\le m<n$, and let $U\in\Gr_{m}(V)$ and $P:=G_{U}$ be such that $\dim(U\cap V_j)=k_j$
for $j=1,2$ with $k_1\le k_2$.
\end{notation}

\begin{proposition}\label{prop:reduction}
Let $\bar{m}=2k-m$, let $\overline{k}_{j}=k-m+k_{j}$, for $j=1,2$, and let $J=(k_{1},k_{2})$ and
$\J=(\overline{k}_{1}, \overline{k}_{2})$.  Then there is a geometry isomorphism from
$\Proj{m,k}{J}(V)$ to $\Proj{\m,k}{\J}(V)$.  In particular, $\Proj{m,k}{J}(V)$ is collinearly
(respectively, concurrently) complete if and only if $\Proj{\m,k}{\J}(V)$ is collinearly
(respectively, concurrently) complete.
\end{proposition}

\begin{proof}
Define a map $f$ from the elements of $\Proj{m,k}{J}(V)$ to the elements of $\Proj{\m,k}{\J}(V)$ by
$f(U):=U^\perp$ if $U\in \Gr_{m}(V)$, and $f(\{V_1,V_2\}):=\{V_1^\perp,V_2^\perp\}$ for a bisection
$\{V_1,V_2\}$ of $V$. It is straight-forward to prove that $f$ is a bijection. Notice also that $f$
sends $\Gr_m(V)$ to $\Gr_{\m}(V)$ (`points to points') and preserves bisections (`lines to lines'),
so in order to show that $f$ is a geometry isomorphism, we only need to show that incidence is
preserved. Suppose we have $U\in\Gr_{m}(V)$ and a bisection $\{V_1,V_2\}$ which are incident in
$\Proj{m,k}{J}(V)$; that is, $(\dim(U\cap V_{1}),\dim(U\cap V_{2}))= (k_{1},k_{2})$ or
$(k_{2},k_{1})$. Suppose without loss of generality that $(\dim(U\cap V_{1}),\dim(U\cap V_{2}))=
(k_{1},k_{2})$. Since $U^\perp \cap V_1^\perp= (U+V_1)^\perp $ and $U^\perp \cap V_2^\perp=
(U+V_2)^\perp$, we have for each $j\in\{1,2\}$,
\begin{align*}
\dim(U^\perp\cap V_{j}^\perp)&=2k-\dim( U+V_j)=2k-\dim U-\dim V_j+\dim(U\cap V_j)\\
&=k-m+k_j=\bar{k}_j.
\end{align*}
Therefore, $U^\perp$ and $\{V_1^\perp,V_2^\perp\}$ are incident in $\Proj{\m,k}{\J}(V)$, and we have
shown that $f$ is a geometry isomorphism. It then follows that $\Proj{m,k}{J}(V)$ is collinearly
(respectively, concurrently) complete if and only if $\Proj{\m,k}{\J}(V)$ is collinearly
(respectively, concurrently) complete.
\end{proof}

Our analysis frequently uses the following property of \emph{diagonal subspaces} of a decomposition
$Y=Y_1\oplus Y_2$. These are subspaces $U$ of $Y$ such that $U\cap Y_j=0$ for $j=1,2$. Clearly a
diagonal subspace has dimension at most $\min\{\dim Y_1,\dim Y_2\}$, and if equality holds, then we
call $U$ a \emph{maximal diagonal subspace}.

\begin{lemma}\label{lem:c12-2diag}
Let $Y=Y_1\oplus Y_2$, and let $1\le r\le \min\{\dim Y_1,\dim Y_2\}$. Then there exist two disjoint
diagonal $r$-subspaces of $Y$ if and only if $(\max\{\dim Y_1,\dim Y_2\},q)\ne (1,2)$.
\end{lemma}

\begin{proof}
Suppose first that $q=2$ and $\dim Y_1=\dim Y_2=1$. Then $Y$ has a basis $\langle e,f\rangle$ where
$Y_1=\langle e\rangle$ and $Y_2=\langle f\rangle$. There is a unique diagonal
$1$-subspace of $Y$, namely $\langle e+f\rangle$. So there do not exist two disjoint diagonal
$r$-subspaces of $Y$ in this case.

Now suppose $(\max\{\dim Y_1,\dim Y_2\},q)\ne (1,2)$.  Suppose $y_1:=\dim Y_1$ and $y_2:=\dim Y_2$,
and assume without loss of generality that $y_1\le y_2$.  Write $Y_1=\langle e_1,\ldots,
e_{y_1}\rangle$ and $Y_2=\langle f_1,\ldots, f_{y_2}\rangle$.  If $q>2$, then there exists a scalar
$a\in\F\setminus\{0,1\}$ and $\langle e_i+f_i:i=1,\ldots, r\rangle$ and $\langle
e_i+af_i:i=1,\ldots, r\rangle$ are disjoint diagonal $r$-subspaces of $Y$. So suppose now that
$q=2$. Then $y_2\ge 2$.  and $Z_1:=\langle e_i+f_i:i=1,\ldots, r\rangle$ and $Z_2:=\langle
e_i+f_{i+1}:i=1,\ldots, r\rangle$ are distinct diagonal $r$-subspaces of $Y$, where it is understood
that $f_{i+1}=f_1$ if $i=y_2$.  We claim that either $Z_1\cap Z_2=\{0\}$ or $r=y_1=y_2$ and $Z_1\cap
Z_2=\la \sum_{i=1}^r(e_i+f_i)\ra$: if $v=\sum_{i=1}^ra_i(e_i+f_i)=\sum_{i=1}^rb_i(e_i+f_{i+1})\in
Z_1\cap Z_2$, then $\sum_{i=1}^r (a_i+b_i)e_i=(\sum_{i=1}^r b_if_{i+1})+(\sum_{i=1}^ra_if_i)\in
Y_1\cap Y_2=\{0\}$, and careful consideration of this equality yields that $a_i=b_i=0$ for all $i$
(so that $v=0$) unless $r=y_1=y_2$ and $a_i=b_i=1$ for all $i$ (so $v=\sum_{i=1}^r(e_i+f_i)$),
proving the claim. Finally suppose that $r=y_1=y_2\ge2$ and replace $Z_2$ by $Z_2':=\la
e_r+f_1+f_2\rangle\oplus\langle e_i+f_{i+1}:1\le i\le r-1 \ra$. A similar computation shows
that $Z_1\cap Z_2'=\{0\}$.
\end{proof}

\subsection{Collinear completeness}\label{sec:gl-c12}

In this section we study collinear completeness of the geometries $\Proj{m,k}{(k_1,k_2)}(V)$ with
$0\le k_1\le k_2\le k_1+k_2\le m\le k$. We later use Proposition~\ref{prop:reduction} for the case
$m>k$. The geometry $\Proj{m,k}{(k_1,k_2)}(V)$ is collinearly complete if and only if each pair of
$m$-subspaces $U_{1}, U_{2}\in \Gr_m(V)$ is incident with at least one bisection. Since all subspace
pairs from $\Gr_m(V)$ which intersect in a subspace of given dimension form an orbit under the
induced $G$-action, we need to check this property for one subspace pair $U_{1}, U_{2}$ for each
possible dimension $t=\dim(U_1\cap U_2)$, namely for each $t$ satisfying $0\le t\le m-1$ (since
$m\le k$).

\subsubsection{The exceptional case $(q,m,k,k_1,k_2)=(2,1,1,0,0)$}\label{sec:c12-special}

The parameters in the heading give rise to an exception in Theorem~\ref{thm:mainC1C2C1},
and so merit special attention.

\begin{proposition}\label{prop:c12-diag}
Let $m$ and $k$ be positive integers such that $1\le m\le k$, and let $V$ be a $2k$-dimensional
vector space. Then $\Proj{m,k}{(0,0)}(V)$ is collinearly complete if and only if $(q,m,k)\neq
(2,1,1)$.
\end{proposition}

\begin{proof}
First let $(m,k,q)= (1,1,2)$, let $\{e, f\}$ be a basis for $V$, and let $U_1=\langle e\rangle$ and
$U_2=\langle f\rangle$. Then a $k$-subspace of $V$ disjoint from both $U_1$ and $U_2$ must be a
diagonal subspace of $U_1\oplus U_2$. However (see Lemma~\ref{lem:c12-2diag}), there is only one
diagonal subspace of $U_1\oplus U_2$, namely $\langle e+f\rangle$, and therefore there is no
bisection incident with both $U_1$ and $U_2$ in $\Proj{1,1}{(0,0)}(V)$, so $\Proj{1,1}{(0,0)}(V)$ is
not collinearly complete when $q=2$.

Now suppose that $(m,k,q)\ne (1,1,2)$. Let $U_1$ and $U_2$ be $m$-subspaces meeting in a
$t$-subspace $T$ (where $t$ is any integer satisfying $0\le t\le m-1$), so we can decompose $U_1,
U_2, V$ as follows: $U_1=P_1\oplus T$, $U_2=P_2\oplus T$, and $$ V=P_1\oplus T\oplus P_2\oplus C $$
where $\dim C=2k-2m+t$. Furthermore, we can decompose $T\oplus C$ as $C_1\oplus C_2$ where $\dim
C_1=\dim C_2=k-m+t$ and $T\subseteq C_1$. We make two applications of Lemma \ref{lem:c12-2diag}:
\begin{enumerate}
\item[(a)] If $(m-t,q)\ne (1,2)$ then there exist disjoint diagonal $(m-t)$-subspaces $D_1$ and $D_2$ of $P_1\oplus P_2$.
\item[(b)] If $(k-m+t,q)\ne (1,2)$ then there exist disjoint diagonal $(k-m+t)$-subspaces $E_1$ and $E_2$ of $C_1\oplus C_2$.
\end{enumerate}
So in particular, if both of these conditions hold, then $V_1:=D_1+E_1$ and $V_2:=D_2+E_2$ are
disjoint $k$-subspaces forming a bisection such that $V_i\cap U_j=0$ for all $i,j$, so $U_1, U_2$
are both incident with $\{V_1,V_2\}$. Thus if both conditions (a) and (b) hold for all values of
$t$, then $\Proj{m,k}{(0,0)}(V)$ is collinearly complete. In particular, this is the case if $q>2$,
by Lemma~\ref{lem:c12-2diag}.

Thus we may assume that $q=2$, and hence that $(m,k)\ne(1,1)$, and that at least one of (a) or (b)
does not hold for some $t$.  Thus $t=m-1$ (if (a) fails) or $t=m-k+1$ (if (b) fails).  Note that, in
either case, $k-m+t\ge 0$.  Also $k\ge 2$ (since $(m,k)\ne (1,1)$ and $m\le k$).  Since
$\dim(C)=2k-2m+t\ge k-m+t\ge0$ there is a decomposition $C=C_3\oplus C_4$ with $\dim(C_3)=k-m+t$
and $\dim(C_4)=k-m$.  Let $D_1$ be a maximal diagonal subspace of $P_1\oplus P_2$ and let
$V_1:=D_1\oplus C_3$.  Then $\dim V_1=k$ and $V_1$ intersects both $U_1$ and $U_2$ trivially.  We
have two cases:

\begin{description}
\item[Case 1: $\mathbf{t=m-1}$.]  If $m\le k-1$, then let $D_2$ be a maximal diagonal subspace of
  $U_2\oplus C_3$ and let $V_2:=D_2\oplus C_4$. Then $V_2$ has dimension $k$ as
  $\dim(D_2)=\min\{m,k-1\}=m$, and by construction $V_2\cap U_2=0$. We also have $V_2\cap U_1=0$
  since $(U_2\oplus C_3)\oplus C_4 = (D_2\oplus C_3)\oplus C_4 = V_2\oplus C_3$ so we have
  $V=P_1\oplus V_2\oplus C_3$ whence $U_1\cap V_2=T\cap V_2\subseteq U_2\cap V_2 = 0$. Thus in this
  case, we have exhibited a bisection $\{V_1,V_2\}$ incident in $\Proj{m,k}{(0,0)}(V)$ with both
  $U_1$ and $U_2$.

So suppose now that $m=k$ and $t=k-1$.  It is sufficient to construct two $k$-subspaces $U_1, U_2$
disjoint from the $k$-subspaces of a given bisection such that $\dim(U_1\cap U_2)=k-1$. So let
$\{V_1, V_2\}$ be a bisection of $V$ and let $U_1$ be a maximal diagonal subspace of $V_1\oplus
V_2$. Let $T$ be a $(k-1)$-subspace of $U_1$, let $u\in U_1\setminus T$, and let $v\in V_1$ lie in
the projection of $T$ onto $V_1$. Define $U_2= T\oplus \la u+v\ra$. Then $U_2$ is a diagonal
$k$-subspace of $V_1\oplus V_2$ such that $U_1\cap U_2=T$.  Thus the bisection $\{V_1,V_2\}$ is
incident in $\Proj{m,k}{(0,0)}(V)$ with both $U_1$ and $U_2$ and $\dim(U_1\cap U_2)=k-1$.

\item[Case 2: $\mathbf{t=m-k+1}$.] Since $t\ge 0$, either $m=k\ge 2$ (and $t=1$) or $m=k-1$ (and
  $t=0$).  Suppose first that $m=k-1$ and $t=0$. By Lemma \ref{lem:c12-2diag} (since $k\ge2$),
  there exists a maximal diagonal subspace $D_2$ of $U_1\oplus U_2$ disjoint from $D_1$. Let
  $V_2:=D_2\oplus C_4$. Then $\dim V_2=k$ as $\dim C_4=k-m=1$.  Clearly $V_2$ intersects $V_1$,
  $U_1$ and $U_2$ trivially. So in this case, we have exhibited a bisection $\{V_1,V_2\}$ incident
  in $\Proj{m,k}{(0,0)}(V)$ with both $U_1$ and $U_2$.

So suppose now that $m=k$ and $t=1$. If $k=2$ then the required subspaces $U_1, U_2$ and bisection
$\{V_1, V_2\}$ were constructed at the end of Case 1. So we may assume that $k\ge3$. As at the end
of Case 1, choose a bisection $\{V_1,V_2\}$ of $V$ and a maximal diagonal subspace $U_1$ of it. Let
$v_1+v_2$ be a non-zero element of $U_1$ with $v_i\in V_i$ for $i=1,2$, let $T=\la v_1+v_2\ra$, and
choose decompositions $U_1=D\oplus T$, $V_1=V_1'\oplus \la v_1\ra$, and $V_2=V_2'\oplus \la
v_2\ra$. Then $D$ is a maximal diagonal subspace of $V_1'\oplus V_2'$. Since $k\ge3$ there exists
(see the proof of Lemma~\ref{lem:c12-2diag}) a maximal diagonal subspace $D'$ of $V_1'\oplus V_2'$
disjoint from $D$. Then $U_2:=D'\oplus T$ is a maximal diagonal subspace of $V_1\oplus V_2$ such
that $U_1\cap U_2=T$. Thus the bisection $\{V_1,V_2\}$ is incident in $\Proj{m,k}{(0,0)}(V)$ with
both $U_1$ and $U_2$ and $\dim(U_1\cap U_2)=1$.
\end{description}

\end{proof}

We end this subsection by proving one more special case, for which the arguments are similar to
those used above. This case will arise in our general analysis in Theorem~\ref{thm:c12:gen-poss}.

\begin{lemma}\label{lem:smallcase}
Let $q=2$, $k_1=0$, $k_2>0$, and $0\le t\le k-1$. Suppose we have one of the following:
\begin{enumerate}
\item[(a)] $k_2=k-1-t$, or
\item[(b)] $k_2=k-1$ and $k\in\{2,3,4\}$.
\end{enumerate}
Then there exist $k$-subspaces $U_1$ and $U_2$ such that $\dim(U_1\cap U_2)=t$, and a bisection of
$V$ incident in $\Proj{k,k}{(0,k_2)}(V)$ with both $U_1$ and $U_2$.  In particular, if $q=2$ and
$2\le k\le 4$, then $\Proj{k,k}{(0,k-1)}(V)$ is collinearly complete.
\end{lemma}

\begin{proof}
Choose $k$-subspaces $U_1, U_2$ with $T=U_1\cap U_2$ of dimension $t$. Suppose first of all that
$k_2=k-1-t$. First we decompose $U_1=T\oplus X_1\oplus \bar{U}_1$, $U_2=T\oplus X_2\oplus
\bar{U}_2$, where $\dim\bar{U}_i=k_2$ and $\dim X_{i}=1$ for $i=1,2$. Since $t\ge 0$, we can
decompose $V=C\oplus (U_1+U_2)$ where $\dim C=t=k-k_2-1$. Note that if $t=0$, then $C$ is the
trivial subspace. Let $D_1$ be a maximal diagonal subspace of $X_1\oplus \bar{U}_2$, and let $D_2$
be a maximal diagonal subspace of $X_1\oplus X_2$. Hence $D_1$ and $D_2$ are disjoint $1$-subspaces.
Let $D_3$ be a maximal diagonal subspace of $C\oplus T$, and note that $\dim D_3=k-k_2-1$, and in
particular, $D_3=0$ if $t=0$. We now choose the bisection to consist of the $k$-spaces $V_1$ and
$V_2$ defined as follows:
\begin{align*}
V_1&:=\bar{U}_1\oplus D_1\oplus D_3\\
V_2&:=\bar{U}_2\oplus D_2 \oplus C.
\end{align*}
It is then straight-forward to see that $U_1\cap V_1=\bar{U_1}$, $U_1\cap V_2=\{0\}$, $U_2\cap
V_2=\bar{U_2}$, and $U_2\cap V_1=\{0\}$, and we have a proof for the first case (a) (and the $t=0$
sub-case for (b)).

Next we suppose that $k_2=k-1$ and $0< t \le k-1$. Let $\{e_1,\ldots e_{2k}\}$ be a basis of $V$
such that $U_1=\langle e_1,\ldots,e_k\rangle$ and $U_2=\langle e_{k-t+1},\ldots,
e_{2k-t}\rangle$. We exhibit in Table \ref{table:nicebisections},
for each $k$ and $t$, a pair of $k$-spaces $V_1$ and $V_2$ forming a
bisection incident with $U_1$ and $U_2$ in $\Proj{k,k}{(0,k_2)}(V)$:
\begin{table}[H]
\begin{center}
\begin{tabular}{c|c|l|l|c|c||c|c}
$k$&$t$&$V_1$&$V_2$&$U_1\cap V_1$&$U_1\cap V_2$&$U_2\cap V_1$&$U_2\cap V_2$\\
\hline
2&1&$\langle
e_3,e_4
\rangle$&$\langle
e_1,e_2+e_4
\rangle$
&$\{0\}$
&$\langle e_1\rangle$
&$\langle e_3\rangle$
&$\{0\}$
\\
3&1&$\langle
e_1+e_6,e_2+e_5,e_4+e_6
\rangle$&$\langle
e_2,e_3,e_4
\rangle$
&$\{0\}$
&$\langle e_2,e_3\rangle$
&$\{0\}$
&$\langle e_3,e_4\rangle$
\\

3&2&$\langle
e_1+e_6,e_4+e_6,e_5
\rangle$&$\langle
e_2,e_3,e_6
\rangle$
&$\{0\}$
&$\langle e_2,e_3\rangle$
&$\{0\}$
&$\langle e_2,e_3\rangle$
\\

4&1&$\langle
e_5,e_6,e_7,e_8
\rangle$&$\langle
e_1,e_2,e_3,e_4+e_8
\rangle$
&$\{0\}$
&$\langle e_1,e_2,e_3\rangle$
&$\langle e_5,e_6,e_7\rangle$
&$\{0\}$
\\

4&2&$\langle
e_1+e_8,e_2+e_7,e_5+e_8,e_6+e_7
\rangle$&$\langle
e_2,e_3,e_4,e_5
\rangle$
&$\{0\}$
&$\langle e_2,e_3,e_4\rangle$
&$\{0\}$
&$\langle e_3,e_4,e_5\rangle$
\\

4&3&$\langle
 e_1+e_8,e_5+e_8,e_6,e_7
\rangle$&$\langle
e_2,e_3,e_4,e_8
\rangle$
&$\{0\}$
&$\langle e_2,e_3,e_4\rangle$
&$\{0\}$
&$\langle e_2,e_3,e_4\rangle$
\\

\hline
\end{tabular}
\caption{Bisections $\{V_1,V_2\}$ incident with $U_1$ and $U_2$ for $k\in\{2,3,4\}$ and $0< t \le
k-1$.}\label{table:nicebisections}
\end{center}
\end{table}
\end{proof}

We prove a technical proposition which leads to a necessary condition for collinear
completeness of $\Proj{m,k}{(k_1,k_2)}(V)$, where $V=V(2k,q)$.

\begin{proposition}\label{thm:c12-gen-imposs}
If $m+k_1-k_2< t\le m-1$, and if $U_1, U_2$ are $m$-subspaces with $U_1\cap U_2$ of dimension $t$,
and $U_1, U_2$ incident in $\Proj{m,k}{(k_1,k_2)}(V)$ with the bisection $V=V_1\oplus V_2$, then
\begin{enumerate}
\item[(a)] for each $i$, $\dim(U_1\cap V_i)=\dim(U_2\cap V_i)$; and
\item[(b)] $3k_2\le k + 1 + m + k_1$.
\end{enumerate}
In particular, if  $\Proj{m,k}{(k_1,k_2)}(V)$ is collinearly complete, then
$3k_2\le k + 1 + m + k_1$.
\end{proposition}

\begin{proof}
Set $\bar m:=m-k_1-k_2$, $\bar t:= m+k_1-k_2+1$, and $T:=U_1\cap U_2$.  The conditions on $t$ imply
that $k_2\ge k_1+2$.  Without loss of generality suppose that, for each $i$, $\dim(U_1\cap
V_i)=k_i$.  Let $\pi_1:V\rightarrow V_1$ be the natural projection. Then $\dim(\pi_1(U_1)) =k_1+\bar
m\ge \dim(\pi_1(T))$. Hence $\dim(T\cap V_2)\ge t - (k_1+\bar m) \ge \bar t - (k_1+\bar m) =
k_1+1$. Since $T\cap V_2\subseteq U_2\cap V_2$, it follows that $\dim(U_2\cap V_2)>k_1$ and hence 
 $\dim(U_2\cap V_2)=k_2$. So part (a) is proved.

Now take $t=\bar t$, and note that the conditions on $t$ hold for this $t$-value. In particular part
(a) holds. Then $k=\dim(V_2)\ge \dim ((U_1\cap V_2)+(U_2\cap V_2)) = 2k_2 -\dim(T\cap V_2)\ge 2k_2
- t = 3k_2-m-k_1-1$. Thus part (b) holds.

We now prove the final assertion. If $k_2\ge k_1+2$, then $t=\bar t \le m-1$ and hence satisfies
the inequalities of the proposition. If $\Proj{m,k}{k_1,k_2}(V)$ is collinearly complete then
suitable $m$-subspaces $U_1, U_2$ exist, and the result follows from part (b). On the other hand, if
$k_2\le k_1+1$, then again $3k_2\le k_2 + 1 + (k_1+k_2) \le k+1+m\le k+1+m+k_1$.
\end{proof}

Now we assume that $m\le k$ and we examine in detail conditions for two $m$-spaces to be incident in
$\Proj{m,k}{(k_1,k_2)}(V)$ with a bisection.

\begin{theorem}\label{thm:c12:gen-poss}
  Let $t$ be an integer satisfying $0\le t\le m-1$, and let $U_1$ and $U_2$ be $m$-subspaces of $V$
  such that $T=U_1\cap U_2$ has dimension $t$.  Then either there exists a bisection incident in
  $\Proj{m,k}{(k_1,k_2)}(V)$ with both $U_1$ and $U_2$, or one of the following holds.
\begin{enumerate}
\item[(i)] $q=2$, $k_1=t=0$, $k=m=k_2+1\ge1$, or

\item[(ii)] $k + 1 + m + k_1<3k_2$ and $t>  m+k_1-k_2$.
\end{enumerate}
\end{theorem}

\begin{proof}We investigate separately various ranges of $t$-values.

\subsubsection{Case $t\le k_1$} For $i=1,2$,  choose $U_{i2}$ to be a $k_i$-subspace
of $U_i$ containing $T$. Then choose a $(k_{3-i})$-subspace $U_{i1}$ of $U_i$ disjoint from $U_{i2}$
(this is possible, since $m\ge k_1+k_2$).  The subspaces $B_j:=U_{1j}+ U_{2j}$, for $j=1,2$, are
disjoint and have dimension $k_1+k_2$ and $k_1+k_2-t$, respectively.  We wish to extend the $B_j$ to
$k$-dimensional subspaces $V_j$ (respectively) such that $V_j\cap U_i=U_{ij}$ (for each $i,j$) and
$V_1\cap V_2=\{0\}$. Let $B=B_1+B_2$, of dimension $2(k_1+k_2)-t$, choose a complement $\bar V$ so
$V=B\oplus \bar V$, and let $\bar{m}:=m-k_1-k_2$.  If $\bar m=0$, then each $U_i$ lies in $B$, and
we simply extend the $B_i$ by subspaces of $\bar V$ of the appropriate dimensions.  Suppose that
$\bar m>0$. Then $U_1$ and $U_2$ project to disjoint $\bar m$-dimensional subspaces $\bar{U_1}$ and
$\bar{U_2}$ of $\bar V$.  It follows from Lemma~\ref{lem:c12-2diag} that, provided $(\dim(\bar
V),q)\ne(2,2)$, we can find a decomposition $\bar V=\bar{V_1}\oplus \bar{V_2}$, with subspaces
$\bar{V_1}, \bar{V_2}$ of dimensions $k-k_1-k_2$ and $k-k_1-k_2+t$ respectively, that are disjoint
from $\bar{U_1}$ and $\bar{U_2}$. Define $V_j=B_j\oplus \bar V_j$ for $j=1,2$. Then $U_i\cap
V_j=U_i\cap B_j = U_{ij}$ for all $i,j$, so both $U_1, U_2$ are incident with the bisection
$V=V_1\oplus V_2$.

Since here $\bar m>0$, the exceptional case $(\dim(\bar V),q)=(2,2)$ arises only if $k=m=k_1+k_2+1$
and $t=0$. (To see this, we have $2=\dim(\bar V)=2k - 2(k_1+k_2)+t$ so $k-k_1-k_2=1-t/2\le 1$, while
on the other hand $1\le \bar m = m-k_1-k_2\le k-k_1-k_2$, and hence $t=0$ and $k=m=k_1+k_2+1$.) In
this case, we may choose the complement $\bar V$ to be $\bar V=\la e_1,e_2\ra$ with $e_1\in U_1,
e_2\in U_2$. Suppose first that $k_2\ge k_1>0$ and choose $f_i\in U_{i1}\setminus\{0\}$ for
$i=1,2$. Then define $\bar V_1:=\la e_1+e_2\ra, \bar V_2:=\la e_1+f_1+f_2\ra$, and
$V_i:=B_i\oplus\bar V_i$ for $i=1,2$. We again have $U_i\cap V_j=U_i\cap B_j = U_{ij}$ for all
$i,j$, so both $U_1, U_2$ are incident with the bisection $V=V_1\oplus V_2$. Now assume that $k_1=0$
but $k_2>0$. In this case write $U_i=U_{ii}\oplus\la e_i\ra$, and let $f_i\in
U_{ii}\setminus\{0\}$. Define $V_i:=U_{ii}\oplus \la f_i+e_{3-i}\ra$ for $i=1,2$. Then $U_i\cap
V_i=U_{ii}$ and $U_i\cap V_{3-i}=0$ for $i=1,2$, so both $U_1, U_2$ are incident with the bisection
$V=V_1\oplus V_2$. This leaves the case $k_1=k_2=t=0$, which is dealt with in
Proposition~\ref{prop:c12-diag}, producing a bisection incident with $U_1, U_2$ with the single
exception $(m,k,q)=(1,1,2)$, as in (i) (with $k_2=0$).

\subsubsection{Case $k_1<t\le 2k_1$} Choose $U_{i2}$ to be a $k_i$-subspace of $U_i$ such that
$U_{12}\subset T$ and $U_{22}\cap T=U_{12}$. (This is possible since $m\ge k_2+k_1\ge k_2+t-k_1$.)
Then choose a $(k_{3-i})$-subspace $U_{i1}$ of $U_i$ disjoint from $U_{i2}$ such that $U_{11}\cap
T=U_{21}\cap T$ has dimension $t-k_1$ (which again is possible since $m\ge k_1+k_2$ and $t-k_1\le
k_1\le k_2$). The subspaces $B_j:=U_{1j}+ U_{2j}$, for $j=1,2$, are disjoint by construction and
have dimensions $2k_1+k_2-t$ and $k_2$, respectively (note $B_2=U_{22}$ and note that
$2k_1+k_2-t<k_1+k_2\le k$). As before we extend the $B_j$ to $k$-dimensional subspaces $V_j$. Let
$B=B_1+B_2$, of dimension $2(k_1+k_2)-t$, choose a complement $\bar V$ so $V=B\oplus \bar V$, and
let $\bar{m}:=m-k_1-k_2$. If $\bar m=0$, then each $U_i$ lies in $B$, and we simply extend the $B_i$
by subspaces of $\bar V$ of the appropriate dimensions. Suppose that $\bar m>0$. Then $U_1$ and
$U_2$ project to disjoint $\bar m$-dimensional subspaces $\bar{U_1}$ and $\bar{U_2}$ of $\bar V$. It
follows from Lemma~\ref{lem:c12-2diag} that, provided $(\dim(\bar V),q)\ne(2,2)$, we can find a
decomposition $\bar V=\bar{V_1}\oplus \bar{V_2}$, with subspaces $\bar{V_1}, \bar{V_2}$ of
dimensions $k-2k_1-k_2+t$ and $k-k_2$ respectively, that are disjoint from $\bar{U_1}$ and
$\bar{U_2}$. Define $V_j=B_j\oplus \bar V_j$ for $j=1,2$. Then $U_i\cap V_j=U_i\cap B_j = U_{ij}$
for all $i,j$, so both $U_1, U_2$ are incident with the bisection $V=V_1\oplus V_2$. We claim that
the exceptional case does not occur here: suppose that $\bar m>0$ and $\dim(\bar V)=2.$ Then
$2=\dim(\bar V)=2k - 2(k_1+k_2)+t$, so $k-k_1-k_2=1-t/2\le 1$, while on the other hand $2=\dim(\bar
V)\ge 2\bar m\ge 2$, so $1=\bar m=m-k_1-k_2\le k-k_1-k_2$. Hence $k=m=k_1+k_2+1$ and $t=0$, which
contradicts our assumption that $t>k_1\ge 0$.

\medskip
From now on we assume that $t>2k_1$. Note that $k_1+k_2\le m$ and so $m+k_1-k_2\ge 2k_1$.

\subsubsection{Case $2k_1<t\le m+k_1-k_2$ } %$0<t- 2k_1\le k-k_2$}
Choose $U_{11}$, $U_{22}$ to be disjoint $k_1$-subspaces of $T$.  Then choose a $k_{2}$-subspace
$U_{12}$ of $U_1$ such that $U_{12}\cap T=U_{22}$, and hence $U_{12}$ is disjoint from $U_{11}$.
This is possible since $\dim(U_1/T)=m-t\ge k_2-k_1$.  Similarly choose a $k_{2}$-subspace $U_{21}$
of $U_2$ such that $U_{21}\cap T=U_{11}$, and hence $U_{21}$ is disjoint from $U_{22}$.  Note that
$U_{12}\cap U_{21}=0$.

Next choose a $(t-2k_1)$-subspace $T_3$ of $T$ such that $T=U_{11}\oplus U_{22}\oplus T_3$, and a
$(2k+2k_1-2k_2-t)$-subspace $\bar V$ of $V$ such that $V=U_{12}\oplus U_{21}\oplus T_3\oplus
\bar{V}$.
Since $T\cap\bar V=0$, $U_1$ and $U_2$ project to disjoint $\bar m$-subspaces $\bar U_1, \bar U_2$
of $\bar V$, where $\bar m=m+k_1-k_2-t\ge0$. Thus $\bar V$ contains $\bar U:=\bar{U_1}\oplus
\bar{U_2}$, and we note that $\dim(\bar V)-2\bar m = 2(k-m)+t\ge t$.  Choose a $(t-2k_1)$-subspace
$T_1$ of $\bar V$ disjoint from $\bar U$ (which is possible since $\dim(\bar V)-2\bar m \ge t\ge
t-2k_1$), and choose a diagonal $(t-2k_1)$-subspace $T_2$ of $T_1\oplus T_3$.  Note that $T_3$ is
then a diagonal subspace of $T_1\oplus T_2$. Also choose disjoint $(k-m+k_1)$-subspaces $S_1, S_2$
of $\bar V$ such that $\bar V=\bar U\oplus T_1\oplus S_1\oplus S_2$.

If $\bar m=0$, define $V_1=U_{21}\oplus T_1\oplus S_1$ and $V_2=U_{12}\oplus T_2\oplus S_2$. These
are disjoint $k$-subspaces so $V=V_1\oplus V_2$ is a bisection, and $U_j\cap V_j= U_{jj}$ of
dimension $k_1$ while $U_j\cap V_{3-j}=U_{3-j,j}$ of dimension $k_2$.
Now assume that $\bar m>0$. Suppose that we are not in the case where both $q=2$ and $r:=\dim(\bar
U_j\oplus S_j) = \bar m + k-m+k_1 = k+2k_1-k_2-t$ is equal to $1$. Then by
Lemma~\ref{lem:c12-2diag}, there exist disjoint diagonal $r$-subspaces $R_1, R_2$ of $(\bar
U_1\oplus S_1) \oplus (\bar U_2\oplus S_2)$. Note that $\bar U\oplus S_1\oplus S_2 = R_1\oplus R_2$,
In this case define $V_1=U_{21}\oplus T_1\oplus R_1$ and $V_2=U_{12}\oplus T_2\oplus R_2$. Again
these are disjoint $k$-subspaces and $U_j\cap V_j= U_{jj}$ has dimension $k_1$ while $U_j\cap
V_{3-j}=U_{3-j,j}$ has dimension $k_2$. The exceptional case is: $q=2$, $r=k+2k_1-k_2-t=1$. Since
$r\ge \dim(\bar U_1)=\bar m\ge 1$ this implies that $1=r=\bar m=m+k_1-k_2-t$. This however implies
that $m=k+k_1$, and since $m\le k$ we obtain $m=k=k_2+t+1, k_1=0, q=2$. Suppose first that also
$k_2=0$. Then by Proposition \ref{prop:c12-diag}, either there exists a bisection incident in
$\Proj{m,k}{(k_1,k_2)}(V)$ with both $U_1$ and $U_2$, or $m=k=1$ and, since $t\le m-1$, also $t=0$
so part (i) holds (with $k_2=0$). Now assume that $k_2>0$. If also $t>0$ then, by Lemma
\ref{lem:smallcase}, there exists a bisection incident in $\Proj{m,k}{(k_1,k_2)}(V)$ with both $U_1$
and $U_2$. On the other hand if $t=0$ then part (i) holds (with $k_2>0$).

\medskip
From now on we assume that $t>m+k_1-k_2$, and of course $t\le m-1$. If also $3k_2> k + 1 + m + k_1$,
then part (ii) holds. Thus we assume that $3k_2\le k + 1 + m + k_1$. This is the last, rather
complicated, case.

\subsubsection{Case $t> m+k_1-k_2$ and  $3k_2\le k + 1 + m + k_1$}

We will show that there exists a bisection $V=V_1\oplus V_2$ incident with both $U_1$ and $U_2$. By
Proposition~\ref{thm:c12-gen-imposs}, for each $i$, we have $\dim(U_1\cap V_i)=\dim(U_2\cap V_i)$.
As we noted above, $m+k_1-k_2\ge 2k_1$, so here $t > 2k_1$.

\subsubsection{Sub-case $t\le k_2$:}
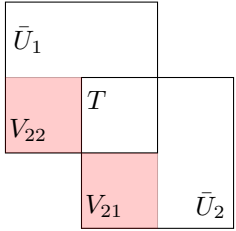
\begin{center}
\begin{tabular}{m{3.5cm}m{12.5cm}}
\begin{tikzpicture}
\draw (0,0) rectangle (2,2);
\draw (1,-1) rectangle (3,1);
\draw [fill=red, opacity=0.2] (0,0) rectangle (1,1);
\draw [fill=red, opacity=0.2] (1,-1) rectangle (2,0);
\node at (0.3,1.5) {$\bar{U}_1$};
\node at (2.7,-0.7) {$\bar{U}_2$};
\node at (1.2,0.7) {$T$};
\node at (0.3,0.3) {$V_{22}$};
\node at (1.3,-0.7) {$V_{21}$};
\end{tikzpicture}
&
\begin{minipage}{12cm}
First we decompose $U_1=T\oplus V_{21}\oplus \bar{U}_1$, $U_2=T\oplus V_{22}\oplus \bar{U}_2$, where
$\dim\bar{U}_i=m-k_2$ and $\dim V_{2i}=k_2-t$. Let $T_2:=V_{21}+V_{22}+T$ and note that $\dim
T_2=2k_2-t$.  Since $m\le k$, we can decompose $V=C\oplus (U_1+U_2)$ where $\dim C=2k-2m+t$.  Now
$k-2k_2+t>k-2k_2+(m+k_1-k_2)$, and since $3k_2\le k + 1 + m + k_1$, we have
$$k-2k_2+t>k-(k+1+m+k_1)+m+k_1=-1.$$
Therefore, we can further decompose $C$ as $C_1\oplus C_2$ where
$\dim C_1=k+2k_2-2m$ and $\dim C_2=k-2k_2+t$.
Let $V_2:=C_2\oplus T_2$ and observe that $\dim(V_2)=k$ and
$\dim(U_i\cap V_2)=\dim(V_{2i}\oplus T)=k_2$ for $i=1,2$.
\end{minipage}
\end{tabular}
\end{center}

Now we construct $V_1$.  Since $k_1\le m-k_2$, there exist $k_1$-subspaces $V_{11}\le \bar{U}_1$ and
$V_{12}\le \bar{U}_2$.  Again, since $k_1+k_2\le m$, we have $\dim C_1=k+2k_2-2m\le k-2k_1$. Then
$C_1\oplus V_{11}\oplus V_{12}$ is a $(k+2(k_1+k_2-m))$-dimensional subspace meeting $U_1$ and $U_2$
in $k_1$-dimensional subspaces, and disjoint from $V_2$. If $m=k_1+k_2$ we take $V_1$ to be this
subspace. On the other hand if $k_1+k_2< m$, then we decompose $\bar{U_i}=W_i+V_{1i}$, where $\dim
W_i=m-k_1-k_2$, choose a diagonal $(m-k_1-k_2)$-subspace $W$ of $W_1\oplus W_2$, and take $V_1$ as
the $k$-subspace $V_1:=V_{11}\oplus V_{12}\oplus C_1\oplus W$. In either case, $\dim(V_1\cap
U_i)=\dim(V_{1i})=k_1$, for $i=1,2$, so the bisection $\{V_1,V_2\}$ is incident with both $U_1$ and
$U_2$.

\subsubsection{Sub-case $t>k_2$:}

First we decompose $U_1+U_2$ as follows:
\begin{equation}\label{U12}
U_1+U_2=\bar{U}_1\oplus T_{13}\oplus T_2\oplus \bar{U}_2
\end{equation}
where $\dim \bar{U}_1=\dim\bar{U}_2=m-t$, $\dim T_{13}=t-k_2$, $\dim T_2=k_2$, and $U_1\cap
U_2=T_{13}\oplus T_2$. Let $C$ be a complement of $U_1+U_2$ in $V$, and decompose it as follows:
$$
C=C_1\oplus C_2\oplus C_3
$$ where $\dim C_1=m-t$, $\dim C_2=k-k_2-m+t$, $\dim C_3=k+k_2-2m+t$.  Each of theses subspaces is
well-defined and non-zero because $\dim C=2k-2m+t$, and
\begin{enumerate}
\item[(i)] $t<m$,
\item[(ii)] $k_2<t$ and $m\le k$ imply that $k-k_2-m+t> 0$, and
\item[(iii)] $(k+k_2)-2m+t\le  2k-2m+t$, and the assumption $t> m+k_1-k_2$ implies that $k+k_2-2m+t> 0$
(since $k+k_2-2m+t> k_1-m+k\ge 0$).
\end{enumerate}
Let $V_2:=C_1\oplus C_2\oplus T_2$ and note that $\dim V_2=k$ and $\dim(V_2\cap U_i)=k_2$ for $i=1,2$.
Now we construct $V_1$. There are two sub-cases.

\subsubsection*{(i) \underline{$t\ge k_1+k_2$}:}

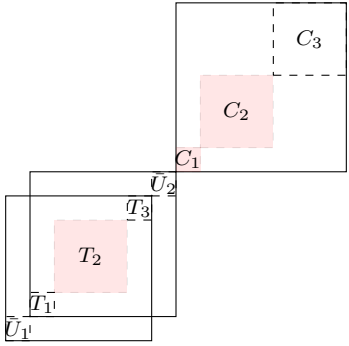
\begin{center}
\begin{tabular}{m{4.8cm}m{\mylen}}
{\footnotesize
\begin{tikzpicture}[scale=0.32]
% example k=7, t =  5, k1=1, k2 =3, m=6
\draw (0,0) rectangle (6,6);
\draw (1,1) rectangle (7,7);
\draw (7,7) rectangle (14,14);

\draw[dashed] (0,0) rectangle (1,1);
\draw[dashed] (6,6) rectangle (7,7);

\draw[dashed] (1,1) rectangle (2,2);
\draw[dashed, fill=red, opacity=0.1] (2,2) rectangle (5,5);
\draw[dashed] (5,5) rectangle (6,6);

\draw[dashed, fill=red, opacity=0.1] (7,7) rectangle (8,8); %m-t = 5
\draw[dashed, fill=red, opacity=0.1] (8,8) rectangle (11,11);  %k-k2-m+t = 3
\draw[dashed] (11,11) rectangle (14,14); %k+k2-2m+t=3

\node at (0.5,0.5) {$\bar{U}_1$};
\node at (6.5,6.5) {$\bar{U}_2$};
\node at (1.5,1.5) {$T_1$};
\node at (3.5,3.5) {$T_2$};
\node at (5.5,5.5) {$T_3$};
\node at (7.5,7.5) {$C_1$};
\node at (9.5,9.5) {$C_2$};
\node at (12.5,12.5) {$C_3$};

\end{tikzpicture}
}
&
\begin{minipage}{\mylen}
Since $t\ge k_1+k_2$, we may decompose $T_{13}$ as $T_1\oplus T_3$ where $\dim T_1=k_1$ and $\dim
T_3=t-k_1-k_2$.  First suppose $k-2m+t\le 0$.  Let $D_{\bar{U}_1,C_1}$, $D_{\bar{U}_2,C_1}$,
$D_{T_2,C_3}$, and $D_{T_3,C_2}$ be maximal diagonal subspaces of $\bar{U}_1\oplus C_1$,
$\bar{U}_2\oplus C_1$, $T_2\oplus C_3$, and $T_3\oplus C_2$ respectively.  Define
$$
V_1:=T_1\oplus D_{\bar{U}_1, C_1}\oplus D_{\bar{U}_2, C_1}\oplus D_{T_2, C_3}\oplus D_{T_3,C_2}.
$$ Now $\dim D_{\bar{U}_i, C_1}=m-t$, for each $i$, and $k-2m+t\le 0$ implies that $\dim D_{T_2,
  C_3}=\min\{k_2, k+k_2-2m+t\}=k+k_2-2m+t$.  Also $\dim D_{T_3,C_2}=t-k_1-k_2$ since $-k_1\le
k-m$. Thus $\dim V_1=k$. Finally, we observe that $V_1$ meets $U_1$ and $U_2$ in $T_1$, and hence
$U_1$ and $U_2$ are each incident with the bisection $\{V_1,V_2\}$.

\end{minipage}
\end{tabular}
\end{center}

Now suppose $k-2m+t> 0$, and decompose $C_3$ as $C_{31}\oplus C_{32}$, where $\dim C_{31}=k_2$ and
$\dim C_{32}=k-2m+t$. Our choice for $V_1$ is similar to before, except we replace ``$D_{T_2,
  C_3}$'' (which has dimension $k_2$ in this case!) with $D_{T_2, C_{31}}\oplus C_{32}$ where
$D_{T_2,C_{31}}$ is a maximal diagonal subspace of $T_2\oplus C_{31}$.  Define
$$
V_1:=T_1\oplus D_{\bar{U}_1, C_1}\oplus D_{\bar{U}_2, C_1}\oplus D_{T_2, C_{31}}\oplus C_{32}\oplus D_{T_3,C_2}.
$$ Note that we still have $\dim(D_{T_2, C_{31}}\oplus C_{32})=k+k_2-2m+t$ and $\dim
D_{T_3,C_2}=t-k_1-k_2$, and so again $\dim(V_1)=k$. Also $V_1\cap U_1=V_1\cap U_2=T_1$, and
therefore, $U_1$ and $U_2$ are each incident with the bisection $\{V_1,V_2\}$.

\subsubsection*{(ii) \underline{$k_2<t<k_1+k_2$:}}

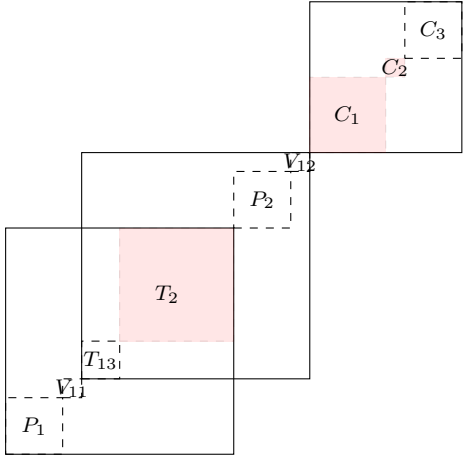
\begin{center}
\begin{tabular}{m{6cm}m{\mylentwo}}
{\footnotesize
\begin{tikzpicture}[scale=0.25]
% example k=12, t =  8, k1=2, k2 =7, m=12
\draw (0,0) rectangle (12,12);
\draw (4,4) rectangle (16,16);
\draw (16,16) rectangle (24,24);

\draw[dashed] (0,0) rectangle (3,3);  % m-k1-k2=3,  P1
\draw[dashed] (3,3) rectangle (4,4);  % V11
\draw[dashed] (12,12) rectangle (15,15); % P2
\draw[dashed] (15,15) rectangle (16,16); % V22
\draw[dashed] (4,4) rectangle (6,6);  % T13
\draw[dashed, fill=red, opacity=0.1] (6,6) rectangle (12,12); %T2

\draw[dashed, fill=red, opacity=0.1] (16,16) rectangle (20,20); %m-t = 4, C1
\draw[dashed, fill=red, opacity=0.1] (20,20) rectangle (21,21);  %k-k2-m+t = 1, C2
\draw[dashed] (21,21) rectangle (24,24); %k+k2-2m+t=3, C3

\node at (1.5,1.5) {$P_1$};
\node at (13.5,13.5) {$P_2$};
\node at (3.5,3.5) {$V_{11}$};
\node at (15.5,15.5) {$V_{12}$};
\node at (5,5) {$T_{13}$};
\node at (8.5,8.5) {$T_2$};
\node at (18,18) {$C_1$};
\node at (20.5,20.5) {$C_2$};
\node at (22.5,22.5) {$C_3$};
\end{tikzpicture}}
&
\begin{minipage}{\mylentwo}
We first refine the decomposition of $U_1+U_2$ in \eqref{U12} by decomposing $\bar{U}_i=P_i\oplus
V_{1i}$ (for $i=1,2$) where $\dim P_i=m-k_1-k_2\ge 0$ and $\dim V_{1i}=k_1+k_2-t$. So
$$
U_1+U_2=P_1\oplus V_{11}\oplus T_{13}\oplus T_2\oplus V_{12}\oplus P_2.
$$

Suppose $k-2m+t\le 0$, so that $\dim C_3 = k+k_2-2m+t\le k_2 = \dim T_2$, and let $D_{T_2,C_3}$ be a
maximal diagonal subspace of $T_2\oplus C_3$, which will therefore have dimension $k_2+k-2m+t$.
Next choose a maximal diagonal subspace $P_3$ of $P_1\oplus P_2$, so $\dim P_3 = m-k_1-k_2$ and
$P_1\oplus P_2 = P_1\oplus P_3$.  Also choose a maximal diagonal subspace $P_4$ of $P_1\oplus
(C_1\oplus C_2)$, and note that $\dim P_4 = \min\{m-k_1-k_2,k-k_2\}=m-k_1-k_2$.  Observe that
$U_i\cap(P_3\oplus P_4)=0$ for each $i$. Define
$$
V_1:=V_{11}\oplus V_{12}\oplus T_{13}\oplus  D_{T_2, C_3}\oplus P_3\oplus P_4
$$
of dimension $\dim V_1=2(k_1+k_2-t)+(t-k_2)+(k_2+k-2m+t)+2(m-k_1-k_2)=k.$
Finally, since $V_1\cap U_i = V_{1i}\oplus T_{13}$ of dimension $k_1$, for
each $i$, both  $U_1$ and $U_2$ are incident with the bisection $\{V_1,V_2\}$.
\end{minipage}
\end{tabular}
\end{center}

Now suppose $k-2m+t> 0$, and decompose $C_3$ as $C_{31}\oplus C_{32}$, where $\dim C_{31}=k_2$ and
$\dim C_{32}=k-2m+t$.  Our construction of $V_1$ is similar to the above, but we replace $
D_{T_2\oplus C_3}$ (which has dimension $k_2$ in this case) with $D_{T_2, C_{31}}\oplus C_{32}$
(which has dimension $k_2+k-2m+t$), where $D_{T_2,C_{31}}$ is a maximal diagonal subspace of
$T_2\oplus C_{31}$.  Define
$$
V_1:=V_{11}\oplus V_{12}\oplus T_{13} \oplus D_{T_2, C_{31}}\oplus C_{32}\oplus P_3\oplus P_4.
$$
Then $\dim V_1=k$ and $V_1\cap U_i = V_{1i}\oplus T_{13}$ of dimension $k_1$, for
each $i$, so both  $U_1$ and $U_2$ are incident with the bisection $\{V_1,V_2\}$.

\end{proof}

\subsubsection{Proof of Theorem~\ref{thm:mainC1C2C1}}

We use the notation from Notation~\ref{notation:c12-0}.  By Lemma~\ref{prop:geom-flag-trans}~(a),
$G=PQP$ if and only if $\Proj{m,k}{(k_1,k_2)}(V)$ is collinearly complete. First we note that, by
Proposition~\ref{prop:c12-diag}, if $(q,m,k,k_1,k_2)=(2,1,1,0,0)$, then $\Proj{m,k}{(k_1,k_2)}(V)$
is not collinearly complete. Thus from now on we assume that $(q,m,k,k_1,k_2)\ne(2,1,1,0,0)$.

Next, suppose that the theorem holds in the case $m\le k$, that is, for these values of $m$,
$\Proj{m,k}{(k_1,k_2)}(V)$ is collinearly complete if and only if $3k_2\le k+1+m+k_1$ and
$(q,m,k,k_1,k_2)\ne(2,1,1,0,0)$. Suppose further that $k<m<2k$, and let $\overline{m}=2k-m$,
$\overline{k}_{j}=k-m+k_{j}$ (for $j=1,2$), and $\J=(\overline{k}_{1}, \overline{k}_{2})$.  Note
that $\overline{m}<k$ and $\overline{k}_{1}\le \overline{k}_{2}$.  By
Proposition~\ref{prop:reduction}, $\Proj{m,k}{(k_1,k_2)}(V)$ is collinearly complete if and only if
$\Proj{\m,k}{\J}(V)$ is collinearly complete. Since $\overline{m} <k$, the latter holds, by our
assumption, if and only if $3\overline{k}_2\le k+1+\overline{m}+\overline{k}_1$, which is equivalent
to $3k_2\le k+1+m+k_1$.
Thus it is sufficient to prove the theorem for the case $m\le k$, and we assume from now on that
$m\le k$ as well as $(q,m,k,k_1,k_2)\ne(2,1,1,0,0)$.

Assume first that $\Proj{m,k}{(k_1,k_2)}(V)$ is collinearly complete. We must prove that $3k_2\le
k+1+m+k_1$. Since $\Proj{m,k}{(k_1,k_2)}(V)$ is collinearly complete, there exist $m$-subspaces
$U_1,$ $U_2$ such that $U_1\cap U_2$ has dimension $t:=m-1$, and $U_1,$ $U_2$ are incident in
$\Proj{m,k}{(k_1,k_2)}(V)$ with a bisection. If $t>m+k_1-k_2$, then by
Proposition~\ref{thm:c12-gen-imposs} we conclude that $3k_2\le k+1+m+k_1$ as required. Suppose then
that $m-1=t\le m+k_1-k_2$. Then $k_2\le k_1+1$, and again we find $3k_2\le k_2+1+(k_1+k_2)\le
k+1+m\le k+1+m+k_1$.

Conversely suppose that $3k_2\le k+1+m+k_1$ and $(q,m,k,k_1,k_2)\ne(2,1,1,0,0)$. Our task is to
prove that $\Proj{m,k}{(k_1,k_2)}(V)$ is collinearly complete. To do this, we must show that, for
each $t$ satisfying $0\le t\le m-1$, and for $m$-subspaces $U_1,$ $U_2$ satisfying $\dim(U_1\cap
U_2)=t$, both $U_1$ and $U_2$ are incident in $\Proj{m,k}{(k_1,k_2)}(V)$ with a bisection.
Suppose first that one of (i) $(q,k_1)\ne(2,0)$, or (ii) $m<k$, or (iii) $m=k\ne k_2=1$ holds. Then
the required subspaces and bisection exist for each $t$ by Proposition~\ref{thm:c12-gen-imposs}, so
$\Proj{m,k}{(k_1,k_2)}(V)$ is collinearly complete.
Thus we may assume that $q=2, k_1=0$, $m=k=k_2+1$, in addition to $3k_2\le k+1+m+k_1$ and
$(q,m,k,k_1,k_2)\ne(2,1,1,0,0)$.  This means that $k_2>0$ (since $(q,m,k,k_1,k_2)\ne(2,1,1,0,0)$)
and the inequality $3k_2\le k+1+m+k_1$ is equivalent to $k\le4$. Thus $k=2,3$ or $4$, and the
hypotheses of Lemma~\ref{lem:smallcase}~(b) all hold. Finally, it follows from
Lemma~\ref{lem:smallcase} that $\Proj{k,k}{(0,k-1)}(V)$ is collinearly complete for these values of
$k$.  This completes the proof of Theorem \ref{thm:mainC1C2C1}.

\subsection{Concurrent completeness}\label{sec:gl-c21}

\subsubsection{Some fundamental results on concurrent completeness}\label{sec:c21-nonexistence}

\begin{lemma}\label{lemma:c21-main-noexist}
Let $m$ and $k$ be positive integers such that $1\le m\le k$, and let $k_1,k_2$ be positive integers
such that $k_1\le k_2$.  If $k_2> m/2$ and $(k,q)\ne (1,2)$, then $\Proj{m,k}{(k_1,k_2)}(V)$ is not
concurrently complete.
\end{lemma}

\begin{proof}
Consider the vector space $V(2,q^k)$ over the finite field of order $q^k$. We can identify the
vectors of $V(2,q^k)$ with the vectors of $V(2k,q)$, which has the effect of ``blowing up''
subspaces: an $e$-dimensional subspace of $V(2,q^k)$ becomes a $ke$-dimensional subspace of
$V(2k,q)$.  So if we take the $q^k+1$ one-dimensional subspaces of $V(2,q^k)$, this ``blowing up''
procedure produces a set $\mathcal{S}$ of mutually disjoint $k$ dimensional subspaces of
$V(2k,q)$. This is known in finite geometry as a \emph{regular} or \emph{Desarguesian spread} of the
ambient projective space. Now take $V_1$, $V_2$, $V_1'$ and $V_2'$ to be four distinct elements of
$\mathcal{S}$; note that $|\mathcal{S}|\ge 4$ precisely when $q^k\ge 3$, that is, $(k,q)\ne (1,2)$.
No $m$-subspace contains a $2k_2$-dimensional subspace since $k_2> m/2$, and so there is no $m$-subspace $U$
that meets both $V_2$ and $V_2'$ in $k_2$-dimensional subspaces. It follows that $\Proj{m,k}{(k_1,k_2)}(V)$ is
not concurrently complete.
\end{proof}

\begin{lemma}\label{lemma:c21-smallerm}
Let $m$, $m'$, and $k$ be positive integers such that $1\le m'\le m\le k$. If $\Proj{m,k}{(0,0)}(V)$
is concurrently complete, then $\Proj{m',k}{(0,0)}(V)$ is concurrently complete.
\end{lemma}

\begin{proof}
Suppose we have two bisections $\{V_1,V_2\}$ and $\{W_1,W_2\}$ of $V$. Since 
$\Proj{m,k}{(0,0)}(V)$ is concurrently complete, we can find a subspace $U$ of dimension $m$
that meets each of $V_1$, $V_2$, $W_1$, and $W_2$ in a trivial subspace.
Take any $m'$-subspace $\bar{U}$ of $U$. Then $\bar{U}$ also has the same property that it
meets $V_1$, $V_2$, $W_1$, and $W_2$ trivially.
Therefore, $\Proj{m',k}{(0,0)}(V)$ is concurrently complete.
\end{proof}

\subsubsection{The case $k=1$ and the case $k_1=m$ }

\begin{lemma}\label{lemma:kequals1}
Suppose $k=1$. Then:
\begin{enumerate}
\item[(i)] $\Proj{1,1}{(0,0)}(V)$ is concurrently complete if and only if $q\ge 4$;   %   q^k\ge 4
\item[(ii)] $\Proj{1,1}{(0,1)}(V)$ is concurrently complete if and only if $q=2$.
\end{enumerate}
\end{lemma}

\begin{proof}
(i) Suppose $k=1$.
Then $\Gr_k(V)$ is simply a projective line containing $q+1$ points. To be concurrently complete, 
for all distinct pairs of points $\{V_1,V_2\}$ and $\{V_3,V_4\}$, there must be a point disjoint from each $V_i$.
This is true if and only if $q\ge 4$.
\medskip

\noindent (ii) Suppose that $(m,k,q)=(1,1,2)$ and notice that the point set of $\Proj{1,1}{(0,1)}(V)$ is simply the
three points of the projective line $\PG(1,2)$.  In this case, it is trivially concurrently complete
since two pairs of points $\{V_1,V_2\}$ and $\{V_1',V_2'\}$ (i.e., bisections) must have an element in common.
By substituting $k=m=1$ into Lemma \ref{lemma:c21-main-noexist}, we see that $\Proj{1,1}{(0,1)}(V)$ is
concurrently complete if and only if $q=2$.
\end{proof}

Applying the previous two lemmas gives us decisive information about  $\Proj{m,k}{(0,m)}(V)$. 

\begin{corollary}\label{cor:c21-main-noexist}
Let $m$ and $k$ be positive integers such that $1\le m\le k$. Then $\Proj{m,k}{(0,m)}(V)$ is
concurrently complete if and only if $(m,k,q)=(1,1,2)$.
\end{corollary}

\subsubsection{Using the `Restricted Movement Criterion'}\label{sec:c21-existence}

To study concurrent completeness of $\Proj{m,k}{J}(V)$, we use the Restricted Movement Criterion
(see \cite[Theorem 1.1(c)]{AFG92}): using Notation~\ref{notation:c12-0}, $G=QPQ$ if and only 
if the set $\Gamma:=U^{Q}$ has
\emph{restricted movement}. In other words,
\begin{align}\label{eq:RMC}
    \Proj{m,k}{J}(V) \text{ is concurrently complete  if and only if } \Gamma^{g}\cap \Gamma\neq \varnothing, \forall g\in G.
\end{align}
Our main result here is Proposition~\ref{prop:c21-main-exist} that shows that if $q^k>4$, then
$\Proj{m,k}{(0,0)}(V)$ is concurrently complete, except possibly when $m=k$, or $q=2$ and $m=k-1$.
Then for the remainder of this section, we will concentrate on the case $m=k$. We need further
notation:

\begin{notation}\label{notation:c12-F-func}
Let $r$, $s$ and $q$ be positive integers such that $r\le s$. Define the integer map $F(r,s,q)$ by
\begin{align*}
F(r,s,q)= \prod_{i=r}^{s}(1-q^{-i}).
\end{align*}
\end{notation}

\begin{remark}\label{rem:c21-diag}
Let $V$, $G$, $P$ and $Q$ be as in Notation~\ref{notation:c12-0}, and let $U$ be an $m$-subspace of
$V$ with $m\ge 2$. Note that if $|\Gamma|>|\Omega|/2$, then $\Gamma$ has restricted movement, and
hence by (\ref{eq:RMC}), $\Proj{m,k}{J}(V)$ is concurrently complete. Here, we give some equivalent
formulas for $|\Gamma|>|\Omega|/2$:
\begin{align*}%\label{eq:c21-F-ineq}
F(k-m+1,k,q)^{2}>\frac{F(2k-m+1,2k,q)}{2},
\end{align*}
or equivalently,
\begin{align}\label{eq:c21-H-ineq}
    H(a,k,q):=\prod_{i=a}^{k}\frac{(1-q^{-i})^{2}}{1-q^{-(k+i)}}>\frac{1}{2},
\end{align}
where $a=k-m+1$. Therefore, (\ref{eq:c21-H-ineq}) is a sufficient condition for obtaining concurrent 
completeness of $\Proj{m,k}{J}(V)$.
\end{remark}

In Lemma~\ref{lem:c21-bound} below, we use the \emph{Maclaurin series} for the \emph{Natural logarithm} function `$\ln$'.
\begin{align}\label{eq:c21-ln}
    \ln(1-x) =-\sum_{j=1}^{\infty} \frac{x^j}{j}, \hbox{ for } -1\le x<1.
\end{align}

\begin{lemma}\label{lem:c21-bound}
Let $k$ and $q$ be positive integers with $q\ge 2$. Then
\begin{enumerate}[(a)]
  \item for every positive integer $i$, we have
  \begin{align}\label{eq:c21-bound1}
    \frac{(1-q^{-i})^{2}}{1-q^{-(k+i)}}> 1-2q^{-i};
  \end{align}
  \item if $a$ is a positive integer such that $2\le a \le k$, then
  \begin{align*}
    H(a,k,q)>1-\frac{2q^{-a}}{1-q^{-1}}-\frac{2q^{-2a}}{(1-2q^{-a})(1-q^{-2})},
  \end{align*}
  where $H(a,k,q)$ is as in (\ref{eq:c21-H-ineq}).
\end{enumerate}
\end{lemma}
\begin{proof}
(a) Note that (\ref{eq:c21-bound1}) holds if and only if $q^{k}(q^{i}-1)^{2}-(q^{k+i}-1)(q^{i}-2)>
  0$, or equivalently, $q^{k}+q^{i}-2>0$. This is true since $\min\{q^{k},q^{i}\}\ge q\ge 2$.

(b) By part (a), we have that
\begin{align*}
    H(a,k,q)= \prod_{i=a}^{k}\frac{(1-q^{-i})^{2}}{1-q^{-(k+i)}}>\prod_{i=a}^{k}(1-2q^{-i}),
\end{align*}
and so $\ln(H(a,k,q))>\ln(\prod_{i=a}^{k}(1-2q^{-i}))=\sum_{i=a}^{k}\ln(1-2q^{-i})$. For $i\ge a\ge
2$, we have that $0<2q^{-i}<1$, and so by (\ref{eq:c21-ln}),
$\ln(H(a,k,q))>-\sum_{i=a}^{k}\sum_{j=1}^{\infty}\frac{(2q^{-i})^{j}}{j}$. Therefore,
\begin{align*}
    \ln(H(a,k,q))&>-\sum_{i=a}^{k}\sum_{j=1}^{\infty}\frac{(2q^{-i})^{j}}{j}\\
    &\ge -\left( \sum_{i=a}^{k}2q^{-i}+\sum_{i=a}^{k}\sum_{j=2}^{\infty}\frac{(2q^{-i})^{j}}{2} \right)\\
    &> -\left( \frac{2q^{-a}}{1-q^{-1}}+2\cdot\sum_{i=a}^{k} \frac{q^{-2i}}{1-2q^{-a}} \right)\\
    &\ge -\left( \frac{2q^{-a}}{1-q^{-1}}+\frac{2}{(1-2q^{-a})} \cdot\sum_{i=a}^{\infty} q^{-2i} \right)\\
   &\ge -\frac{2q^{-a}}{1-q^{-1}}-\frac{2q^{-2a}}{(1-2q^{-a})(1-q^{-2})}.
\end{align*}
Since $H(a,k,q)=\exp(\ln(H(a,k,q)))>1+\ln(H(a,k,q))$, we have that $H(a,k,q)> 1-
\frac{2q^{-a}}{1-q^{-1}}- \frac{2q^{-2a}}{(1-tq^{-a})(1-q^{-2})}$.
\end{proof}

\begin{proposition}\label{prop:c21-main-exist}
Let $m$ and $k$ be positive integers such that $1\le m\le k$ and $q^k>4$. Then $\Proj{m,k}{(0,0)}(V)$ is
concurrently complete unless one of the rows of Table~\ref{tbl:c21-excep-prop} holds.
\begin{table}[H]
\caption{Unresolved cases for concurrent completeness of  $\Proj{m,k}{(0,0)}(V)$ when $m\le k$.}\label{tbl:c21-excep-prop}
\begin{center}
\begin{tabular}{ll}
\hline\noalign{\smallskip}
$q$ & Conditions on $m$ and $k$ \\
\noalign{\smallskip}\hline\noalign{\smallskip}
 $2$  & $m=k-1\ge 2$ \\
      & $m=k\ge 3$ \\ \hline
   $3$& $m=k\ge 2$\\ \hline
   $4$& $m=k\ge 3$ \\
\noalign{\smallskip}\hline
\end{tabular}
\end{center}
\end{table}

\end{proposition}

\begin{proof}
By Remark~\ref{rem:c21-diag}, $\Proj{m,k}{(0,0)}(V)$ is concurrently complete if
$H(a,k,q)>\frac{1}{2}$, where $H(a,k,q)$ is as in (\ref{eq:c21-H-ineq}). Suppose first that $m=1$.
Then $a:=k-m+1=k$, and so
\begin{equation*}
    H(a,k,q)=\frac{(1-q^{-k})^{2}}{(1-q^{-2k})}=1-\frac{2}{q^{k}+1}.
\end{equation*}
Since $q^k>3$, we have $H(a,k,q)>\frac{1}{2}$, and therefore,
$\Proj{m,k}{(0,0)}(V)$ is concurrently complete.

Suppose now that $2\le m\le k$. Let $a:=k-m+1$, and so by  Lemma~\ref{lem:c21-bound}, we have that
\begin{align}\label{eq:c21-bound3}
    H(a,k,q)>(1-2q^{-a})\cdot \left( 1- \frac{2q^{-(a+1)}}{1-q^{-1}}- \frac{2q^{-2(a+1)}}{(1-2q^{-(a+1)})(1-q^{-2})} \right).
\end{align}

Assume first $a=1$, or equivalently, $m=k$. If $m=k= 2$, then
\begin{align*}
    H(1,2,q)&=\frac{q(q-1)^2 (q+1)}{(q^2+1) (q^2+q+1)}
\end{align*}
It turns out that $H(1,2,q)>\frac{1}{2}$ if and only if $q\ge 4$. If $m=k\ge 3$, then for $q\ge 5$,
we have that $1-2q^{-1}\ge \frac{3}{5}$, $\frac{2q^{-2}}{1-q^{-1}}\le \frac{1}{10}$, and
$\frac{2q^{-4}}{(1-2q^{-3})(1-q^{-2})}\le \frac{1}{276}$. Thus by (\ref{eq:c21-bound3}), we have
that $H(1,k,q)>(\frac{3}{5})(1-\frac{1}{10}-\frac{1}{276})=\frac{1237}{2300}$, and hence
$H(1,k,q)>\frac{1}{2}$ unless $q=2,3,4$.

Assume now $a=2$, that is to say, $m=k-1$. Note that $k\ge 3$ as $m\ge 2$. If $q\ge 3$, then by
(\ref{eq:c21-bound3}), we have that
$H(2,k,q)>(\frac{7}{9})(1-\frac{1}{9}-\frac{1}{300})=\frac{5579}{8100}$. Hence
$H(2,k,q)>\frac{1}{2}$ unless $q=2$.

Assume finally $a\ge 3$. Then $m\le k-2$, and so $k\ge m+2\ge 4$. Using $q\ge2$, we have, by
(\ref{eq:c21-bound3}), that $H(a,k,q)>\frac{3}{4}\left( 1-\frac{1}{4}-\frac{1}{84}
\right)=\frac{31}{56}$, and hence $H(a,k,q)>\frac{1}{2}$.

\end{proof}

In the next section, we will resolve the cases in Table~\ref{tbl:c21-excep-prop} for which $m=k$.

\subsubsection{The geometries  $\Proj{k,k}{(0,0)}(V)$ }

\begin{lemma}\label{lem:c21-smallcase}
Let $q^k\le 4$. Then $\Proj{k,k}{(0,0)}(V)$ is concurrently complete if and only if $(q,k)=(4,1)$.
\end{lemma}

\begin{proof}
The case $k=1$ follows from Lemma \ref{lemma:kequals1}. So suppose $k=2$. Then $q=2$ since $q^k\le
4$. Consider $V$ as the row-space $\mathbb{F}_2^4$ and let $\{e_1,e_2,e_3,e_4\}$ be the canonical
basis. Let $V_1=\langle e_1, e_2\rangle$, $V_2=\langle e_3,e_4\rangle$, $V_3=\langle e_2+e_4,
e_3\rangle$, and $V_4=\langle e_1, e_2+e_3\rangle$. Note that $\{V_1,V_2\}$ and $\{V_3,V_4\}$ are
bisections of $V$. These $2$-spaces cover $11$-vectors of $V$. The remaining vectors are $e_1+e_4$,
$e_1+e_3$, $e_1+e_3+e_4$, $e_1+e_2+e_4$, and $e_1+e_2+e_3+e_4$. There is no $2$-space whose nonzero
vectors are contained within this set of five vectors, so it is impossible to find a $2$-space $U$
disjoint from the components of these two bisections. Therefore, $\Proj{k,k}{(0,0)}(V)$ is not
concurrently complete.
\end{proof}

Notice in the proof above that we did not seek to prove the impossibility of a simpler case, where the $V_i$
are pairwise disjoint. In fact, this strategy does not work for good reasons (see below).

\begin{lemma}\label{lem:fourskew}
Suppose $q^k\ge 4$.
Let $\pi_1$, $\pi_2$, $\pi_3$, $\pi_4$ be $k$-subspaces of $V(2k,q)$ that pairwise intersect trivially. Then
there exists a fifth $k$-subspace $\sigma$ meeting each $\pi_i$ trivially.
\end{lemma}

\begin{proof}
If $k=1$, then clearly the result follows since the projective line arising from $V(2,q)$ has at
least five elements. So suppose $k\ge 2$, and assume by way of contradiction, that there is no
$k$-subspace $\sigma$ meeting each $\pi_i$ trivially. Then $\{\pi_1,\pi_2,\pi_3,\pi_4\}$ is what is
known in finite projective geometry as a \emph{maximal partial spread}\footnote{A \emph{partial
spread} is a set of mutually disjoint $k$-dimensional subspaces.}, and so the underlying set of
points ($1$-dimensional subspaces) forms a \emph{blocking set with respect to $k$-spaces} (we are
using algebraic dimensions here). By a result of Heim \cite{Heim:1996fk}, if $B$ is a nontrivial
blocking set with respect to $k$-spaces, and if $q>2$, then
\[
|B|\ge (q^{k+1} - 1)/(q - 1) +q^{k-1} r_2(q)      % Check (k=1): q+1+r_2(q), yes!
\]
where $q+1+r_2(q)$ is the size of the smallest nontrivial blocking set of $\PG(2,q)$.
Since $r_2(q)\ge \sqrt{q}$ (by a result of Bruen \cite{Bruen:1970fk}), we have
(for $q>2$)
\[
4\cdot \frac{q^k-1}{q-1} \ge \frac{q^{k+1} - 1}{q - 1} +q^{k-1}\sqrt{q},
\]
which implies that $4q^k\ge q^{k+1}+q^{k-1}\sqrt{q}(q-1)$ and hence $4q-q^2-\sqrt{q}(q-1)\ge 0$;
which is a contradiction for $q>2$. So since $\{\pi_1,\pi_2,\pi_3,\pi_4\}$ is a partial spread, but
is not maximal, we can find a larger partial spread containing it. Therefore, there exists a
$k$-subspace $\sigma$ meeting each $\pi_i$ trivially.

Now suppose that $q=2$. Without loss of generality, we can assume that $\pi_1$, $\pi_2$, $\pi_3$,
$\pi_4$ are rowspaces for the following $k\times 2k$ matrices:
\begin{align*}
\pi_1:&\begin{bmatrix} I&O\end{bmatrix},&\pi_2:&\begin{bmatrix} O&I\end{bmatrix},\\
\pi_3:&\begin{bmatrix} I&I\end{bmatrix},&\pi_4:&\begin{bmatrix} I&A\end{bmatrix},
\end{align*}
where $I$ and $O$ are the $k\times k$ identity and zero matrices, respectively, and $A$ is some
matrix that is invertible (since $\pi_4$ is disjoint from $\pi_1$) and $A+I$ is invertible (as
$\pi_4$ is disjoint from $\pi_3$). Let $\sigma$ be the row space of $\begin{bmatrix} I&A^{-1}
\end{bmatrix}$.
Now $A^{-1}$ also has the property that $A^{-1}+I$ is invertible since $A^{-1}+I=A^{-1}(I+A)$.
The same argument shows that $\sigma$ is disjoint from $\pi_1$, $\pi_2$, and $\pi_3$. 
Since
\[
A^{-1}(A+I)^2=A^{-1}(A^2+2A+I)=A^{-1}(A^2+I)=A+A^{-1},
\]
we see that $A+A^{-1}$ is invertible, and hence, $\sigma$ is disjoint from $\pi_4$.
\end{proof}

\begin{lemma}\label{lem:fournotskew}
Let $k\ge 2$ and $\pi_1$, $\pi_2$, $\pi_1'$, $\pi_2'$ be $k$-subspaces of $V(2k,q)$ such that
\begin{enumerate}
\item[(i)] $\pi_2\cap \pi_1'\ne\{0\}$,
\item[(ii)] $\pi_1\cap\pi_2=\{0\}$,
\item[(iii)] $\pi_1'\cap\pi_2'=\{0\}$.
\end{enumerate}
If $\Proj{k-1,k-1}{(0,0)}(V(2k-2,q))$ is concurrently complete, then there is a $k$-space $\sigma$ 
disjoint from each $\pi_i$, $\pi_i'$.
\end{lemma}

\begin{proof}
Since $\pi_2\cap \pi_1'$ is nontrivial, $\langle \pi_2,\pi_1'\rangle$ is contained in a hyperplane
$\Sigma$. We know from (ii) and (iii) that $\pi_1$ and $\pi_2'$ are not contained in $\Sigma$ and so
they must meet $\Sigma$ in $(k-1)$-subspaces.

We will show now that there exists a $(k-1)$-subspace $\tau$ of $\Sigma$ disjoint from $\pi_2$,
$\pi_1'$, $\pi_1\cap \Sigma$, and $\pi_2'\cap\Sigma$. By (i), there exists a $1$-dimensional
subspace $\alpha$ of $\pi_2\cap \pi_1'$. Consider the quotient space $\Sigma/\alpha$ of dimension
$2k-2$. Recall that the quotient map is $U\mapsto \langle U,\alpha\rangle/\alpha$ and so the images
of $\{\pi_1\cap\Sigma,\pi_2\}$ and $\{\pi_1',\pi_2'\cap\Sigma\}$ yield a pair of bisections of
$\Sigma/\alpha$. Since $\Proj{k-1,k-1}{(0,0)}(V(2k-2,q))$ is concurrently complete, there exists a
$(k-1)$-subspace $\bar{\tau}$ of $\Sigma/\alpha$ disjoint from the images of $\pi_1\cap\Sigma$,
$\pi_2$, $\pi_1'$, and $\pi_2'\cap\Sigma$. The preimage of $\bar{\tau}$ is a $k$-subspace on
$\alpha$ that is disjoint from $\pi_1\cap\Sigma$ and $\pi_2'\cap\Sigma$, and meeting $\pi_1'$ and
$\pi_2$ precisely in the subspace $\alpha$. Hence we can take a $(k-1)$-subspace $\tau$ of this
preimage disjoint from $\alpha$ to obtain a $(k-1)$-subspace of $\Sigma$ disjoint from $\pi_2$,
$\pi_1'$, $\pi_1\cap \Sigma$, and $\pi_2'\cap\Sigma$.

Now there are $q^{2k}-q^{2k-1}=q^{2k-1}(q-1)$ vectors not in $\Sigma$, and each $k$-space containing
$\tau$ and not contained in $\Sigma$ contains exactly $q^{k-1}(q-1)$ such vectors. Thus there are
$q^k$ such subspaces and we denote them $s_1,\ldots,s_{q^k}$. Let $T = (\pi_1\cup \pi_2')\setminus
\Sigma$ and note that $|T|\le 2q^{k-1}(q-1)$. If $s_i\cap T$ is non-empty, say contains a vector
$x$, then it contains each non-zero scalar multiple of $x$. Hence $|s_i\cap T|\ge q-1$. Since each
of the $s_i$ contains $\tau$, the sets $s_i\cap T$ are pairwise disjoint, and hence the number of
$i$ with $s_i\cap T$ non-empty is at most $T\le 2q^{k-1}$. If $q>2$ then $2q^{k-1}<q^k$, and hence
there exists $i$ with $s_i\cap T$ empty. Therefore, for $q>2$, there exists a $k$-space that is
disjoint from each of the four subspaces $\pi_1$, $\pi_2$, $\pi_1'$, $\pi_2'$.

Suppose now that $q=2$. The argument above yields the required subspace unless each of the $s_i\cap
T$ is non-empty, so we may assume that this is the case. Thus, in particular, $q^k\le |T|\le
2q^{k-1}$, that is, $|T|=2^k$, which implies that $\pi_1\setminus \Sigma$ and $\pi_2'\setminus
\Sigma$ are disjoint. Moreover, from the previous paragraph we have $\sum_{i=1}^{2^k}|s_i\cap T|\le
2^k$, and hence for each $i$, $|s_i\cap T|=1$ and the singleton $s_i\cap T$ lies in either $\pi_1$
or $\pi_2'$. If $s_i\cap \pi_1\subset \pi_1$ then $\langle s_i,\pi_1\rangle = \langle \tau, \pi_1
\rangle$ is a hyperplane of $V$. Similarly if $s_i\cap \pi_2'\subset \pi_2'$ then $\langle
s_i,\pi_2'\rangle = \langle \tau, \pi_2' \rangle$ is a hyperplane of $V$, and exactly one occurs for
each $i$. Consider the quotient vector space $\ol{V}$ obtained by factoring out $\tau$. Now
$\dim(\ol{V})=k+1$, the image $\ol{\Sigma}$ of $\Sigma$ is a hyperplane, and consequently
$|\ol{V}\setminus\ol{\Sigma}|=2^k$. Moreover each of the $k$-spaces $s_i$ maps to a 1-space
$\ol{s_i}$ spanned by the image $\ol{s_i\cap T}$ of $s_i\cap T$, and these are $2^k$ pairwise
distinct 1-spaces. It follows that $\ol{V}\setminus\ol{\Sigma}$ is equal to the set $\{\ol{s_i\cap
T} \mid 1\le i\le 2^k\}$, and hence that every element of $\ol{V}\setminus\ol{\Sigma}$ is
contained in either the image $\ol{\pi_1}$ (of $\pi_1$) or the image $\ol{\pi_2'}$ (of $\pi_2'$).
The space $\ol{V}$ has the structure of an affine space $\mathsf{AG}(k+1,2)$, and each of
$\ol{\Sigma}, \ol{V}\setminus\ol{\Sigma}$, $\ol{\pi_1}$ and $\ol{\pi_2'}$ is an affine hyperplane.
Thus $\ol{\pi_1}\cap \ol{\Sigma}$, $\ol{\pi_2'}\cap \ol{\Sigma}$, $\ol{\pi_1}\setminus \ol{\Sigma}$
and $\ol{\pi_2'}\setminus \ol{\Sigma}$ are all affine $(k-1)$-subspaces, and $\ol{\pi_1}\cap
\ol{\Sigma}$, $\ol{\pi_2'}\cap \ol{\Sigma}$ is parallel to $\ol{\pi_1}\setminus \ol{\Sigma}$,
$\ol{\pi_2'}\setminus \ol{\Sigma}$, respectively. In addition we have shown that
$\ol{\pi_1}\setminus \ol{\Sigma}$ and $\ol{\pi_2'}\setminus \ol{\Sigma}$ are parallel to each other.
It follows that $\ol{\pi_1}\cap \ol{\Sigma}$ and $\ol{\pi_2'}\cap \ol{\Sigma}$ are parallel, and
this is a contradiction, since these both contain the zero vector of $\ol{V}$. Thus this final case
does not arise, and the proof is complete.
\end{proof}

We put Lemma \ref{lem:fournotskew} and Lemma \ref{lem:fourskew} together:

\begin{lemma}\label{lem:C2C1C2induction}
Let $k\ge 2$. If $\Proj{k-1,k-1}{(0,0)}(V(2k-2,q))$ is concurrently complete, then
$\Proj{k,k}{(0,0)}(V(2k,q))$ is concurrently complete.
\end{lemma}

\begin{proposition}\label{propn:main_mk}
The geometry $\Proj{k,k}{(0,0)}(V(2k,q))$ is concurrently complete if and only if $(q,k)\ne(2,1), (3,1), (2,2)$.
\end{proposition}

\begin{proof}
Note that by Lemma \ref{lem:C2C1C2induction}, we need only prove 
concurrent completeness for small $k$, since the result
will then hold by induction on $k$. Hence, it suffices to show the following:
\begin{enumerate}[(a)]
\item $\Proj{1,1}{(0,0)}(V(2,q))$ is concurrently complete if and only if $q\ge4$.
\item $\Proj{2,2}{(0,0)}(V(4,3))$ and $\Proj{3,3}{(0,0)}(V(6,2))$ are each concurrently complete.
\end{enumerate}
Statement (a) is true by Lemma \ref{lemma:kequals1}. We now prove statement (b). Suppose first that
$q=3$ and $k=2$. We resolved this case by computer, since there are only $130$ two-dimensional
subspaces of $V(4,3)$, and $5265$ possible bisections. The stabiliser $G$ in $\GL(4,3)$ of a
bisection $\pi$ has $15$ orbits on the remaining bisections, and the orbit lengths are $24$, $64^2$,
$72$, $96$, $144$, $192$, $288^2$, $384$, $576^3$, $768$, $1152$ (exponents denote multiplicities).
From each of these orbits we chose a bisection, say $\pi'$, and we found a $2$-space $\sigma$
disjoint from each member of each of the bisections $\pi,\pi'$, so that $\sigma$ is collinear with
both $\pi$ and $\pi'$ in $\Proj{2,2}{(0,0)}(V(4,3))$. Therefore, $\Proj{2,2}{(0,0)}(V(4,3))$ is
concurrently complete.

Finally suppose that $q=2$ and $k=3$. We resolved this case also by computer. The stabiliser $G$ in
$\GL(6,2)$ of a bisection $\pi$ has $34$ orbits on the remaining bisections, and the orbit lengths
are $ 98^1 $, $ 336^1 $, $ 441^1 $, $ 588^2 $, $ 784^1 $, $ 1176^1 $, $ 1568^1 $, $ 1764^1 $, $
3528^2 $, $ 4032^1 $, $ 7056^4 $, $ 9408^4 $, $ 14112^6 $, $ 18816^1 $, $ 28224^6 $. Again for a
representative bisection $\pi'$ from each orbit, we found a $3$-space $\sigma$ disjoint from each
member of each of the bisections $\pi,\pi'$. Therefore, $\Proj{3,3}{(0,0)}(V(6,2))$ is concurrently
complete.
\end{proof}

\subsection{Proof of Theorem~\ref{thm:mainC2C1C2} }

Theorem~\ref{thm:mainC2C1C2}(i) follows from Lemmas~\ref{lemma:c21-main-noexist} and~\ref{lemma:kequals1}(ii). 
Now we prove part (ii). 
Let $V=V(2k,q)$ and let $m$ be a positive integer such that $m\le k$. % and $(q,k)\ne(2,m+1)$. 
Suppose first that $q^k> 4$. Then by Proposition \ref{prop:c21-main-exist}, $\Proj{m,k}{(0,0)}(V)$ is concurrently
complete unless one of the rows of Table~\ref{tbl:c21-excep-prop} holds.
Let us first consider the case that $(q,k)\ne(2,m+1)$,
in particular $m=k$. By Proposition~\ref{propn:main_mk}, $\Proj{k,k}{(0,0)}(V)$ is concurrently complete
for all $q^k>4$. Thus we may suppose that $q^k\le 4$. For these small values, we have $m=k\le  2$, since
$q=k=m+1=2$ is not a possibility because of our assumption  $(q,k)\ne(2,m+1)$. By Lemma~\ref{lem:smallcase}, 
$\Proj{k,k}{(0,0)}(V)$ is concurrently complete if and only if $(q,k)=(4,1)$, giving the exceptions 
$(q,k)=(2,1), (3,1)$, and $(2,2)$. 

For the case $q=2$ and $m=k-1$, we will apply Lemma \ref{lemma:c21-smallerm}. 
We have just shown above, that for $q=2$, $\Proj{k,k}{(0,0)}(V)$ is concurrently complete if and only if
$k>3$. Hence, $\Proj{k-1,k}{(0,0)}(V)$ is concurrently complete when $k\ge 3$. Therefore,
we need only resolve one open case; namely when $m=1$ and $k=2$.
Suppose that $\Proj{1,2}{(0,0)}(V)$ is not concurrently complete. Then there are two bisections
 $\{V_1,V_2\}$ and $\{V_1',V_2'\}$ of $V(4,2)$ such that there is no one-dimensional subspace disjoint
 from all four subspaces $V_1$, $V_2$, $V_1'$, and $V_2'$. However, at most $4\cdot 3=12$
(projective) points can be covered by the lines $V_1$, $V_2$, $V_1'$, and $V_2'$, and there
are $2^4-1=15$ points of $\PG(3,2)$; a contradiction.
Therefore, $\Proj{1,2}{(0,0)}(V)$ is concurrently complete.

This completes the proof.

\section{Acknowledgements}
The first author is grateful for support of an International Postgraduate Research Scholarship
(IPRS). Much of the work of this paper arises from the PhD Thesis of the first author, although we
should point out that since the completion of the thesis, some improvements to the results have been
made by the three authors.  This project forms part of Australian Research Council Discovery Grant
DP0770915 and the second author is supported by Discovery Grant DP0984540. The third author is
supported by an Australian Research Council Federation Fellowship FF0776186.

%\bibliographystyle{abbrv}
%\bibliography{references}

\def\cprime{$'$}

\end{document}